\documentclass[reqno, 12pt]{amsart} 
\usepackage{array} 
\usepackage{amsmath}
\usepackage{amsfonts}
\usepackage{amssymb}
\usepackage{mathrsfs} 
\usepackage{enumerate}
\usepackage{amsthm}
\usepackage{verbatim}
\usepackage{amsmath, amscd}
\usepackage[usenames,dvipsnames]{color}
\usepackage[shortlabels]{enumitem}
\usepackage{bm}
\usepackage{xy}
\xyoption{all}
\usepackage{color}
\usepackage{marvosym}
\usepackage{amsbsy}
\usepackage{hyperref}
\usepackage{tikz, tikz-cd}
\usepackage{marginnote}
\usetikzlibrary{matrix,arrows,backgrounds}
\usepackage[backgroundcolor=white]{todonotes}
\usepackage{hyperref}

\newtheorem{theorem}{theorem}[section]
\newtheorem{Theorem}[theorem]{Theorem}
\newtheorem{Lemma}[theorem]{Lemma}
\newtheorem{notation}[theorem]{Notation}
\newtheorem{proposition}[theorem]{Proposition}
\newtheorem{Corollary}[theorem]{Corollary}

\theoremstyle{definition}
\newtheorem{definition}[theorem]{Definition}
\newtheorem{remark}[theorem]{Remark}

\newcommand{\sslash}{\mathbin{/\mkern-6mu/}}

\newcommand{\op}[1]{\operatorname{#1}}

\newcommand{\newterm}{\textsf}

\newcommand{\dbcoh}[1]{\operatorname{D}^{\operatorname{b}}(\operatorname{coh }#1)}

\newcommand{\dsing}[1]{\operatorname{D}_{\operatorname{sg}}(#1)}

\renewcommand{\div}{\operatorname{div}}

\newcommand{\I}{\mathcal I}
\newcommand{\J}{\mathcal J}
\newcommand{\cone}{\operatorname{Cone}}
\newcommand{\Hom}{\operatorname{Hom}}
\newcommand{\tot}{\operatorname{tot}}
\newcommand{\coker}{\op{coker}}

\def\Z{\op{\mathbb{Z}}}
\def\C{\op{\mathbb{C}}}
\def\R{\op{\mathbb{R}}} 

\def\Q{\op{\mathbb{Q}}}

\def\O{\op{\mathcal{O}}}

\def\P{\op{\mathbb{P}}}

\def\I{\op{\mathcal{I}}}
\def\J{\op{\mathcal{J}}}

\def\spec{\operatorname{Spec}}
\def\conv{\op{Conv}}

\title[A derived equivalence of Libgober-Teitelbaum and Batyrev-Borisov]{A derived equivalence of the Libgober-Teitelbaum and the Batyrev-Borisov mirror constructions}

\author[Malter]{Aimeric Malter}
\address{
  \begin{tabular}{l}
   Aimeric Malter \\
   \hspace{.1in} University of Birmingham, School of Mathematics \\
   \hspace{.1in} Edgbaston, Birmingham B15 2TT,  United Kingdom \\
   \hspace{.1in} Email: {\bf ahm933@student.bham.ac.uk} \\
  \end{tabular}
}

\begin{document}
%Hello World!

\maketitle

\bibliographystyle{alpha}

\begin{abstract}
    In this paper we study a particular mirror construction to the complete intersection of two cubics in $\P^5$, due to Libgober and Teitelbaum. Using variations of geometric invariant theory and methods of Favero and Kelly, we prove a derived equivalence of this mirror to the Batyrev-Borisov mirror of the complete intersection. 
\end{abstract}

\section{Introduction}

\noindent Libgober and Teitelbaum \cite{LT94} proposed a mirror to a Calabi-Yau complete intersection $V_\lambda$ of two cubics in $\P^5$ defined as the zero locus for the two polynomials \[Q_{1,\lambda}=x_0^3+x_1^3+x_2^3-3\lambda x_3x_4x_5,\quad 
Q_{2,\lambda}=x_3^3+x_4^3+x_5^3-3\lambda x_0x_1x_2.\]
Their proposed mirror $W_{LT,\lambda}$ is a (minimal) resolution of singularities of the variety $V_{LT,\lambda}$ with defining equations $Q_{1,\lambda}, Q_{2,\lambda}$ but in the quotient space $\P^5/G_{81}$, where $G_{81}$ is a specified order 81 subgroup of $PGL(5,\C)$. They showed topological evidence that $V_{\lambda}$ and $W_{LT,\lambda}$ are a mirror pair, proving  on the level of Euler characteristics that $\chi(V_\lambda)=-\chi(W_{LT,\lambda})$. In \cite{FR18}, Filipazzi and Rota verify a state space isomorphism between the two Calabi-Yau varieties by providing an explicit mirror map.

Batyrev and Borisov in \cite{BB96} introduced a mirror construction for Calabi-Yau intersections in Fano toric varieties using polytopes, showing  mirror duality for $(1,q)$-Hodge numbers.
This mirror construction agrees with constructions by Green-Plesser \cite{GP90} and Berglund-H\"ubsch \cite{BH92} for Fermat hypersurfaces. However, the Batyrev-Borisov mirror to two cubics in $\P^5$ differs from the one given above by Libgober and Teitelbaum.

In this paper, we establish a connection between the mirrors of Libgober-Teitelbaum and Batyrev-Borisov for two cubics in $\P^5$ in the context of Homological Mirror Symmetry, using variations of geometric invariant theory (VGIT). 
In particular, we show that the bounded derived category of coherent sheaves of the Libgober-Teitelbaum mirror is derived equivalent to that of a complete intersection $Z\subseteq X_\nabla$ in the Batyrev-Borisov mirror family.
Note that there exists a toric stack $\mathcal{X}_\nabla$ with coarse moduli space $X_\nabla$ (see \ref{coxconstruction} for the toric stack construction and \ref{sec:BB} for the fan associated to this toric stack).
On the level of stacks, we will prove the following result.

\begin{Theorem}
\label{Thm:MyThmInIntro}
Let $\lambda\in\C$ such that $\lambda^6\ne 0,1$. 
Consider the two polynomials \[
p_{1,\lambda}=x_0^3x_6^3+x_1^3x_7^3+x_2^3x_8^3-3\lambda x_3x_4x_5x_6x_7x_8,\]
\[
p_{2,\lambda}=x_3^3x_9^3+x_4^3x_{10}^3+x_5^3x_{11}^3-3 \lambda x_0x_1x_2x_9x_{10}x_{11}.
\] Let $\mathcal{Z}_\lambda=Z(p_{1,\lambda},p_{2,\lambda})\subseteq \mathcal{X}_\nabla$ and $\mathcal{V}_{LT,\lambda}=Z(Q_{1,\lambda},Q_{2,\lambda})\subseteq[\P^5/G_{81}]$. Then
 \[ \dbcoh{\mathcal{V}_{LT,\lambda}}\simeq \dbcoh{\mathcal{Z}_\lambda}.\]
\end{Theorem}

This result is expected in the context of Kontsevich's Homological Mirror Symmetry Conjecture. As both $\mathcal{V}_{LT,\lambda}$ and $\mathcal{Z}_\lambda$ are conjectured to be (homological) mirrors of the complete intersection of two cubics, we expect their corresponding derived categories to be equivalent to each other and to the Fukaya category of the zero locus $Z(Q_{1,\lambda}, Q_{2,\lambda}) \subseteq \P^5$. 

There has been work to unify various (toric) mirror constructions \cite{ACG16, BB96, Bat93, BH92, Cla17} in the literature via derived equivalence \cite{FK17, DFK18}. This paper adds a new construction to this that has been elusive in the past. In particular, this is the first application of partial compactifications in VGIT quotients to prove the equivalence of derived categories for complete intersections, and not hypersurfaces, for Calabi-Yau varieties. 

We start by giving some background on the mathematical tools necessary to prove Theorem~\ref{Thm:MyThmInIntro} in Section \ref{sec:Back}. This includes a short introduction to the relevant tools in toric geometry, the Batyrev-Borisov mirror construction, and VGIT quotients as outlined in \cite{FK19}. In Section \ref{sec:linking}, we then study the link between the Batyrev-Borisov mirror construction and the mirror given by Libgober-Teitelbaum, proving Theorem $\ref{Thm:MyThmInIntro}$. 

\subsection{Acknowledgements}
I would like to thank my advisor Tyler Kelly and Daniel Kaplan for some stimulating discussions on the results of the paper as well as for their reviews of early versions of this paper. I would further like to thank the other PhD students at the University of Birmingham for their input on finding a simpler proof to Lemma \ref{Lem:inequalitybounds}. I would also like to thank the referee for their suggestions on how to improve the paper.

This project was supported by the Engineering and Physical Sciences Research Council (EPSRC) under Grant EP/L016516/1.

\section{Background}\label{sec:Back}

\noindent In this section we give the necessary background on the Batyrev-Borisov mirror construction, the Libgober-Teitelbaum construction, and the tools used to connect those two. All the varieties considered in this paper will be defined over the complex numbers.
More detailed expositions can for example be found in \cite{BN07}, \cite{CLS}, \cite{FK17} and 
\cite{LT94}.
\subsection{The Cox construction for toric stacks}\label{coxconstruction}

Let $M$ be a lattice of rank $d$ and $N$ its dual lattice, with the pairing \[
\langle\ ,\ \rangle: M\times N\rightarrow \Z.
\]
\noindent
We extend this to a pairing between $M_{\R} := M \otimes_{\Z}\R$ and $N_{\R} := N \otimes_{\Z}\R$ in the natural way.

To associate a variety $X_\Sigma$ to a fan $\Sigma$, we can use the \newterm{Cox construction} (see $\S 5$ of \cite{CLS}).
Start by noting that each ray $\rho$ of the fan $\Sigma$ corresponds to a divisor $D_\rho$ on $X_\Sigma$ (see $\S 4$ of \cite{CLS}). Then we have the following exact sequence:
\begin{equation}
\label{eq:Divisor}
0\rightarrow M\stackrel{\iota}{\rightarrow} \bigoplus_{\rho\in\Sigma(1)}\Z D_\rho\rightarrow \coker\iota \rightarrow 0,
\end{equation}
where $\iota(m):= \operatorname{div}(\chi^m)=\sum_{\rho\in\Sigma(1)}\langle m,u_\rho\rangle D_\rho$.

We will write $\Z^{\Sigma(1)}:=\bigoplus_{\rho\in\Sigma(1)}\Z D_\rho$.
Since $\C^\ast$ is a divisible group and hence an injective $\Z$-module, the functor $\Hom_{\Z}(-,\C^\ast)$ is exact, so applying it to \eqref{eq:Divisor} yields the exact sequence:
\begin{equation}
\label{eq:CoxConpre}
1\rightarrow \Hom_{\Z}(\coker\iota,\C^\ast)\rightarrow\Hom_{\Z}(\Z^{\Sigma(1)},\C^\ast)\rightarrow\Hom_{\Z}(M,\C^\ast)\rightarrow 1.
\end{equation}
Define 
\begin{equation}\label{def:G}
G_{\Sigma}:=\Hom_{\Z}(\coker\iota,\C^\ast).
\end{equation}
Note that $\Hom_{\Z}(\Z^{\Sigma(1)},\C^\ast)\simeq (\C^\ast)^{\Sigma(1)}$ and $\Hom_{\Z}(M,\C^\ast)\simeq T_N$, where $T_N$ is the torus of the variety.
Hence we may rewrite $(\ref{eq:CoxConpre})$ as  \begin{equation}
\label{eq:CoxCon}
1\rightarrow G_{\Sigma}\rightarrow (\C^\ast)^{\Sigma(1)}\rightarrow T_N\rightarrow 1.
\end{equation}

When describing $G_{\Sigma}$ explicitly, the following lemma is useful.
\begin{Lemma}[Lemma 5.1.1(c) in \cite{CLS}]
\label{Lem:5.1.1CLS}
Let $G_\Sigma \subseteq (\C^\ast)^{\Sigma(1)}$ be as in $(\ref{eq:CoxCon})$. Given a basis $e_1,\dots,e_n$ of $M$, we have
\begin{equation*}
 G_{\Sigma}=\left\{(t_\rho)\in(\C^\ast)^{\Sigma(1)}\ \Bigg\vert \ \prod_\rho t_\rho^{\langle e_i,u_\rho\rangle}=1\text{ for }1\le i\le n\right\}.
\end{equation*}
\end{Lemma}

\noindent We now have both an affine space $\C^{\Sigma(1)}$ and a group $G_\Sigma$, which can be shown to be reductive, thus only further require an exceptional set $Z$  in order to construct the toric variety $X_\Sigma$ as a geometric quotient. For each ray $\rho\in\Sigma(1)$, introduce a variable $x_\rho$ and consider the \newterm{total coordinate ring} of $X_\Sigma$, 
\[
S:=\C[x_\rho\ \vert\ \rho\in\Sigma(1)].
\]
For each cone $\sigma\in\Sigma$, let $x^{\hat{\sigma}}=\prod_{\rho\not\in\sigma(1)} x_\rho$. We define the \newterm{irrelevant ideal}
\[
B(\Sigma)=\langle x^{\hat{\sigma}}\ \vert \ \sigma\in\Sigma\rangle\subseteq S.
\]
Since $\tau\preceq \sigma$, we have that $x^{\hat{\tau}}$ is a multiple of $x^{\hat{\sigma}}$. Thus, we only need to consider maximal cones to generate the irrelevant ideal. 
Define $Z(\Sigma)=Z(B(\Sigma))\subseteq\C^{\Sigma(1)}$. 
We then have:
\begin{Theorem}[Theorem 5.1.11 in \cite{CLS}]
\label{Thm:CoxCon}
Let $X_\Sigma$ be a toric variety without torus factors, associated to a fan $\Sigma$. Then \[X_\Sigma\simeq (\C^{\Sigma(1)}\setminus Z(\Sigma))\sslash G_\Sigma.\]
\end{Theorem}

Most of the discussion to follow happens on the level of \newterm{stacks}, so we define the toric stacks relevant for us here.
\begin{definition}
Let $\Sigma$ be a fan. Define the \newterm{Cox fan} $\operatorname{Cox}(\Sigma)\subseteq\R^{\Sigma(1)}$ to be
\[
\operatorname{Cox}(\Sigma):=\{\cone(e_\rho\ \vert\ \rho\in\sigma)\big\vert\sigma\in\Sigma\}.
\]
\end{definition}
\noindent Denote by $n$ the number of rays in the fan $\Sigma$. Then the Cox fan of $\Sigma$ is a subfan of the standard fan corresponding to the toric variety $\mathbb{A}^n$. Thus, $U_\Sigma:=X_{\operatorname{Cox}(\Sigma)}$ is an open subset of $\mathbb{A}^n$. Consider the group $G_\Sigma$ as defined in Equation~\eqref{def:G}. 
 \begin{definition}
 We call $U_\Sigma$ the \newterm{Cox open set} associated to $\Sigma$ and define the \newterm{Cox stack} associated to $\Sigma$ to be \[
 \mathcal{X}_\Sigma :=\left[ U_\Sigma/G_\Sigma\right]
 \]
 \end{definition}
\noindent In the smooth and orbifold cases, we have the following result relating $\mathcal{X}_\Sigma$ to $X_\Sigma$.
\begin{Theorem}[\cite{FMN10}]
\label{Thm:FK18Stacks}
If $\Sigma$ is simplicial, then $\mathcal{X}_\Sigma$ is a smooth Deligne-Mumford stack with coarse moduli space $X_\Sigma$. When $\Sigma$ is smooth (or equivalently $X_\Sigma$ is smooth) $\mathcal{X}_\Sigma\cong X_\Sigma$.
\end{Theorem}

\subsection{Polytopes and the Batyrev-Borisov construction}
We now define \newterm{reflexive polytopes} and \newterm{nef partitions}. We can use them to introduce the Batyrev-Borisov duality, following \cite{BN07, BB96}.

\begin{definition}
A \newterm{polytope} $\Delta$ in $M_{\R}$ is a convex hull of a finite set of points in $M_{\R}$. If this finite set can be chosen to only consist of lattice points of $M$, we call $\Delta$ a \newterm{lattice polytope}.
\end{definition}
\begin{definition}
Let $\Delta$ be a full dimensional lattice polytope in $M_{\R}$ with $0$ an interior lattice point. Its \newterm{dual polytope} $\Delta^\vee$ is given by \[
\Delta^\vee:=\{n\in N_{\R}\ \vert\  \langle m,n\rangle\geq -1\ \forall m\in\Delta\}
\]
We call $\Delta$ \newterm{reflexive} if the dual polytope is also a lattice polytope.
\end{definition}
\noindent Given a lattice polytope $\Delta$, we can associate a toric variety to it by considering its normal fan $\Sigma_\Delta$ with its corresponding toric variety $X_{\Sigma_\Delta}$.

The polytope $\Delta$ corresponds to the anticanonical divisor of $X_{\Sigma_\Delta}$ in that the lattice points of $\Delta$ correspond to the global sections of the anticanonical divisor. This in turn allows one to construct a Calabi-Yau hypersurface in $X_{\Sigma_\Delta}$ by considering the zero-section of the global section; however, we want to construct Calabi-Yau complete intersections. To do so, we must construct a nef partition of the polytope $\Delta$.

\begin{definition}
Let $\Delta\subseteq M_{\R}$ be a reflexive lattice polytope. A \newterm{nef partition of length $r$} of $\Delta$ is a Minkowski sum decomposition $\Delta=\Delta_1+\dots+\Delta_r$ where $\Delta_1,\dots,\Delta_r$ are lattice polytopes with $0\in \Delta_i$. 
\end{definition}

Consider a reflexive polytope $\Delta\subseteq M_{\R}$ with nef partition $\Delta=\Delta_1+\dots+\Delta_r$.
Then, for $1\le j\le r$, we define \[
\nabla_j:=\{n\in N_{\R}\ \vert\ \langle m,n\rangle\geq -\delta_{ij}\text{ for all } m\in\Delta_i,\text{ for } 1\le i\le r\}.
\]
We note that these polytopes are all lattice polytopes, and define the polytope $\nabla$ as their Minkowski sum $\nabla:=\nabla_1+\dots+\nabla_r$. We call $\nabla_1,\dots,\nabla_r$ the \newterm{dual nef partition} to $\Delta_1,\dots,\Delta_r$.

To understand the statement of Batyrev-Borisov duality, we note that a lattice polytope $\Delta$ corresponds to a $d$-dimensional Gorenstein Fano toric variety $X_\Delta$.
Each of the polytopes $\Delta_i$ corresponds to a divisor $D_i$ on $X_\Delta$. The nef partition $\Delta=\Delta_1+\dots+\Delta_r$ decomposes the anticanonical sheaf $\O(-K_{X_\Delta})$ as tensor product $\bigotimes_{i=1}^r\O_{X_\Delta}(D_i)$. Now the lattice points inside the $\Delta_i$ correspond to global sections of these line bundles. Taking the zero-sets of such sections, we can associate to each polytope a family of hypersurfaces. By intersecting these, a nef partition corresponds to a family of $(d-r)$-dimensional Calabi-Yau complete intersections in $X_\Delta$. Similarly, the dual nef partition $\nabla=\nabla_1+\dots+\nabla_r$ gives a family of $(d-r)$-dimensional Calabi-Yau complete intersections in $X_\nabla$.

\begin{remark}\label{Rem:BB}
The generic complete intersection in the family associated to the dual nef partition $\nabla_1,\dots,\nabla_r$ may be singular.

 In \cite{Bat93}, Batyrev formulates the original construction in a way that fixes this problem. In this case, one uses a \newterm{maximal projective crepant partial desingularization} (\newterm{MPCP-desingularization}), which reduces to a combinatorial manipulation of the normal fan to $\nabla$. 

For every maximal cone of the normal fan, we choose a regular triangulation of it. Therefore, all maximal cones should contain exactly the minimal number of rays dictated by the dimension, since a triangulation uses simplices. Doing this for all maximal cones gives exactly a maximal projective triangulation. When speaking of $X_\nabla$ we will thus think of a MPCP-desingularization of the variety associated to the normal fan of $\nabla$, obtained in this way.
\end{remark}

 Batyrev and Borisov prove the following result, showing that their construction produces topological mirror duality for $(1,q)$-Hodge numbers.
\begin{Theorem}[Theorem 9.6 in \cite{BB96}] 
\label{Thm:9.6inBB}
Let $V$ be a Calabi-Yau complete intersection of $r$ hypersurfaces in $\P^d$ and $d-r\ge3$ and  $\widehat{W}$ be a MPCP-desingularization of the Calabi-Yau complete intersection $W\subseteq X_\nabla$. Then\[
h^q(\Omega^1_{\widehat{W}})=h^{d-r-q}(\Omega^1_V)\text{ for } 0\le q\le d-r.
\]
 \end{Theorem}

\subsection{Toric vector bundles and GIT quotients}

We first discuss how to construct toric vector bundles. Recall that a Cartier divisor $D=\sum_\rho a_\rho D_\rho$ on a toric variety $X_\Sigma$ corresponds to the line bundle $\mathcal{L}=\O_{X_\Sigma}(D)$, which is the sheaf of sections of a rank 1 vector bundle $\pi:V_\mathcal{L}\rightarrow X_\Sigma$. The variety $V_\mathcal{L}$ is toric and $\pi$ is a toric morphism. This is shown by directly constructing the fan of $V_\mathcal{L}$ in terms of $\Sigma$ and $D$, which we do now.
 Given a cone $\sigma\in\Sigma$, set
\[
\tilde{\sigma}=\cone{((0,1),(u_\rho,-a_\rho)\ \vert\ \rho\in\sigma(1))}.
\]

\noindent Then $\tilde{\sigma}$ is a strongly convex rational polyhedral cone in $N_{\R}\times\R$ for all cones $\sigma\in\Sigma$. Now let $\Sigma\times D$ be the collection consisting of cones $\tilde{\sigma}$ for $\sigma\in\Sigma$ and their faces. This is a fan in $N_{\R}\times \R$ and the projection $\overline{\pi}:N\times\Z\rightarrow N$ is compatible with $\Sigma\times D$ and $\Sigma$, thus inducing a toric morphism \[
\pi:X_{\Sigma\times D}\rightarrow X_\Sigma
\]

\begin{proposition}[Proposition 7.3.1 in \cite{CLS}]\label{prop:7.3.1CLS}
$\pi:X_{\Sigma\times D}\rightarrow X_\Sigma$ is a rank 1 vector bundle whose sheaf of sections is $\O_{X_{\Sigma}}(D)$.
\end{proposition}

\noindent The variety $X_{\Sigma\times D}$ is sometimes also denoted by $X_{\Sigma,D}$.

For decomposable vector bundles of rank higher than 1, we can repeatedly apply Proposition \ref{prop:7.3.1CLS} to construct the total space of the vector bundle, following \cite{FK18}. 
 Taking $r$ torus-invariant Weil divisors $D_i=\sum_{\rho\in\Sigma} a_{i\rho}D_\rho$, we define 
\[
\sigma_{D_1,\dots,D_r}:=\cone\left(\{u_\rho-a_{1\rho}e_1-\dots-a_{r\rho}e_r\ |\  \rho\in\sigma(1)\}\cup\{e_i | i\in\{1,\dots,r\}\}\right)\subset N_{\R}\oplus \R^r.
\]
Let $\Sigma_{D_1,\dots,D_r}$ be the fan generated by the cones $\sigma_{D_1,\dots,D_r}$ and their proper faces, and call $\mathcal{X}_{\Sigma,D_1,\dots,D_r}$  the associated stack. We obtain the following result.

\begin{proposition}[Proposition 4.13 in \cite{FK18}]\label{Prop:repeatProp7.3.1} Let $D_1,\dots,D_r$ be divisors on $X_\Sigma$. There is an isomorphism of stacks \[
\mathcal{X}_{\Sigma,D_1,\dots,D_r}\cong \tot\left(\O_{\mathcal{X}_\Sigma}(D_i)\right).
\]
\end{proposition}

Geometric invariant theory (GIT), developed by Mumford, is a powerful tool in modern algebraic geometry. We will here discuss the toric version of it, following $\S 14$ of \cite{CLS}.

\noindent Roughly speaking, GIT deals with ways to take almost geometric quotients of spaces by some reductive groups acting on them. As a model for this, recall the Cox construction in \textsection\ref{coxconstruction}. It gives a toric variety as almost geometric quotient $X_\Sigma\simeq (\C^{\Sigma(1)}\setminus Z(\Sigma))\sslash G_{\Sigma}$. Fundamentally, we start with $\C^{\Sigma(1)}$ and remove a special Zariski closed subset in order to obtain an almost geometric quotient. GIT provides the machinery to do so, but the way is often not unique. The subsets that are  removed  depend on a choice  of stability parameterised by a choice of line bundle. The different choices can give different quotients which are birational.

In GIT, deciding which points are removed is done by a lifting of the $G$-action on $\C^r$ to the rank 1 trivial vector bundle $\C^r\times\C\rightarrow \C^r$. Define the character group of $G$ to be \[
\widehat{G}=\{\chi:G\rightarrow\C^\ast\vert \chi\text{ is a homomorphism of algebraic groups}\}.
\]
A character $\chi\in \widehat{G}$ then gives the action of $G$ on $\C^r\times\C$ defined by\[
g\cdot (p,t)=(g\cdot p,\chi(g)t),\ \ g\in G,\ (p,t)\in\C^r\times\C. 
\]
This lifts the $G$-action on $\C^r$ and furthermore all possible liftings arise this way.

Let $\mathcal{L}_\chi$ or $\O(\chi)$ denote the sheaf of sections of $\C^r\times\C$ with this $G$-action. It is called the \newterm{linearised line bundle} with character $\chi$. For $d\in \Z$, the tensor product $\O(\chi)^{\otimes d}$ is the linearised line bundle with character $\chi^d$. Note that, if one forgets the $G$-action, then $\O(\chi)\simeq \O_{\C^r}$ as line bundles on $\C^r$. Thus, a global section $s\in\Gamma(\C^r,\O(\chi))$ can be written as\begin{eqnarray*}
&& s:\C^r\rightarrow \C^r\times \C\\
&& p\mapsto (p,F_s(p)),\;
\end{eqnarray*}
for some unique $F_s\in\C[x_1,\dots,x_r]$.

\begin{definition}\label{Def:stablepoints}
Fix $G\subseteq (\C^\ast)^r$ and $\chi\in\widehat{G}$, with linearised line bundle $\O(\chi)$.
Given a global section $s$ of $\O(\chi)$, we denote \[
(\C^r)_s:=\{p\in\C^r\ \vert\ s(p)\neq0\}
\]
This is an affine open subset of $\C^r$, as $s(p)\neq 0$ means $F_s(p)\neq 0$. Furthermore, $G$ acts on $(\C^r)_s$ when $s$ is $G$-invariant. We define:
\begin{enumerate}
    \item $p\in\C^r$ is \newterm{semistable} with respect to $\chi$ if there exist $d>0$ and $s\in\Gamma(\C^r,\O(\chi^d))^G$ such that $p\in(\C^r)_s$.
    \item $p\in\C^r$ is \newterm{stable} with respect to $\chi$ if there exist $d>0$ and $s\in\Gamma(\C^r,\O(\chi^d))^G$ such that $p\in(\C^r)_s$, the isotropy subgroup $G_p$ is finite, and all $G$-orbits in $(\C^r)_s$ are closed in $(\C^r)_s$.
    \item The set of all semistable (resp. stable) points with respect to $\chi$ is denoted $(\C^r)^{ss}_\chi$ (resp. $(\C^r)^s_\chi$).
\end{enumerate}
\end{definition}

Given a group $G\subseteq(\C^\ast)^r$ and $\chi\in\widehat{G}$, we next need to define the GIT quotient $\C^r\sslash_\chi G$. Consider the graded ring $R_\chi=\bigoplus_{d=0}^\infty\Gamma(\C^r,\O(\chi^d))^G$.

\begin{definition}\label{Def:GITquot}
For $G\subseteq(\C^\ast)^r$ and $\chi\in\widehat{G}$, the \newterm{GIT quotient} $\C^r\sslash_\chi G$ is \[
\C^r\sslash_\chi G=\op{Proj}(R_\chi).
\]
\end{definition}

An important property of GIT quotients is that in principle, this is the same as taking the quotient of $(\C^r)^{ss}_\chi$ under the action of $G$.

\begin{proposition}[Proposition 14.1.12.c) in \cite{CLS}]\label{Prop:14.1.12inCLS}
For $G\subseteq(\C^\ast)^r$ and $\chi\in\widehat{G}$, the GIT quotient $\C^r\sslash_\chi G$ is a good categorical quotient of $(\C^r)^{ss}_\chi$ under the action of $G$, i.e. $\C^r\sslash_\chi G\simeq (\C^r)^{ss}_\chi\sslash G$.
\end{proposition}

Theorem 14.2.13 of \cite{CLS} shows, using a  polyhedron associated to the character $\chi$, that the GIT quotient $\C^r\sslash_\chi G$ is a toric variety.

\subsection{GKZ Fans}\label{sec:secfans}

Let $G \subseteq (\mathbb{C}^*)^r$. Studying the GIT quotient $\C^r\sslash_\chi G$ as $\chi$ varies gives rise to the \newterm{GKZ fan} of a toric variety, which has the structure of a \newterm{generalised fan}.

\begin{definition}\label{Def:GenFan}
A \newterm{generalised fan} $\Sigma$ in $N_{\R}$ is a finite collection of cones $\sigma\subseteq N_{\R}$ such that:
\begin{enumerate}
    \item Every $\sigma\in\Sigma$ is a rational polyhedral cone.
    \item For all $\sigma\in\Sigma$, each face of $\sigma$ is also in $\Sigma$.
    \item For all $\sigma_1,\sigma_2\in\Sigma$, the intersection $\sigma_1\cap\sigma_2$ is a face of each.
\end{enumerate}
\end{definition}
This agrees with the usual definition of a fan, with the exception that cones are not necessarily strongly convex. Consider the cone $\sigma_0=\bigcap_{\sigma\in\Sigma}\sigma$. It has no proper faces and is thus a subspace of $N_{\R}$. We consider the lattice $\overline{N}=N/(\sigma_0\cap N)$. To associate a toric variety for the generalised fan $\Sigma$, one constructs the fan $\overline{\Sigma}$ where each cone comes from a cone of $\Sigma$ quotiented by $\sigma_0$. This is a fan in the usual sense, and hence we can associate a toric variety to it as usual. Then $X_{\Sigma}:=X_{\overline{\Sigma}}$.

We will now discuss the notion of a GKZ fan, following  both \cite{CLS} and \cite{FK19}.
Consider a toric variety $X$. It can be written as a GIT quotient $(\C^r\setminus Z)\sslash_\chi G$.  Recall the character group $\widehat{G}$ of $G$. 
Each choice of character $\chi\in\widehat{G}$ determines an open subset $U_\chi:=(\C^r)^{ss}_\chi$, the semi-stable locus of $X$ with respect to $\chi$. Several different characters can give the same semi-stable locus.
Thinking of the vector space $\Hom(\widehat{G},T_N)\otimes_{\Z}\Q$ as parameter space for linearisations, we investigate where the semi-stable locus $U_\psi$ is the same as $U_\chi$ for a given character $\chi$.  It turns out that dividing the vector space into chambers where $U_\chi$ remains the same gives the space a natural fan structure. This fan-structure $\Sigma_{GKZ}$ is called the \newterm{GKZ fan}. Maximal cones are called \newterm{chambers} and codimension one cones are called \newterm{walls}.

Consider an arbitrary fan $\Sigma$, we can construct the GKZ fan as follows. Take the group $G=G_\Sigma\subseteq (\C^\ast)^r$ acting on $X_\Sigma$ to be the group in Equation \eqref{def:G}.
There is a well-known bijection between chambers of GKZ fans and regular triangulations of a certain set of points, constructed as follows.
In the general setting, apply $\Hom(-,\C^\ast)$ to the sequence
\[
0\rightarrow G\xrightarrow{i_G}(\C^\ast)^r\xrightarrow{proj}\coker(i_G)\rightarrow 0
\]
to obtain the sequence
\[
\Hom(\coker(i_G),\C^\ast)\xrightarrow{\widehat{proj}}\Z^r\xrightarrow{\widehat{i_G}}\Hom(G,\C^\ast)\rightarrow 0.
\]
Let $\nu_i(G)$ be the element of $\Hom(\coker(i_G),\C^\ast)^\vee$ given by composing $\widehat{proj}$ with the projection of $\Z^r$ onto its $i^{\text{th}}$ factor.
Compare this sequence with the sequence ~\eqref{eq:Divisor}. We in fact reversed the process of obtaining \eqref{eq:CoxCon} from \eqref{eq:Divisor}. Starting with the correct group acting on the space, we thus recover the map corresponding to $\iota$ as $\widehat{proj}$. Hence, the $\nu(G)$ correspond to the primitive generators $u_\rho$ of the rays of $\Sigma$.
Then the set we will triangulate is the convex hull of the set $\nu(G)=\{\nu_1(G),\dots,\nu_r(G)\}$.

\begin{Theorem}[Proposition 15.2.9 in \cite{CLS}]
\label{Thm:chambers2nd}
There is a bijection between chambers of the GKZ fan for the action of $G$ on $\C^r$ and regular triangulations of the set $\conv(\nu(G))$. In particular, there are only finitely many chambers of the GKZ fan.
\end{Theorem}

Thus we can enumerate the chambers of the GKZ fan, say by $\sigma_1,\dots,\sigma_k$. For any of those chambers, we can choose a character in its interior and consider the semi-stable locus with respect to it. As this locus does not depend on the choice of character, but solely on the choice of chamber, denote the open affine associated to chamber $\sigma_p$ by $U_p$. By the above theorem, it will also correspond to a specific triangulation $\mathcal{T}_p$ of $\conv(\{\nu_1(G),\dots,\nu_r(G)\})$.

\subsection{Categories of singularities and some results on the equivalence of derived categories.}\label{sec:FK}

In this section, we introduce the categories of singularities (as outlined in \cite{O09}) and their equivalences to derived categories through VGIT,  reviewing $\S 4$ of \cite{FK19}.

 Let $X$ be a variety and $G$ an algebraic group acting on $X$ (on the left).
 
\begin{definition}
\label{Def:SingCat}
An object of $\dbcoh{[X/G]}$ is called \newterm{perfect} if it is locally quasi-isomorphic to a bounded complex of vector bundles. We denote the full subcategory of perfect objects by $\operatorname{Perf}([X/G])$. The Verdier quotient of $\dbcoh{[X/G]}$ by $\operatorname{Perf}([X/G])$ is called the \newterm{category of singularities} and denoted 
\[
\dsing{[X/G]}:=\dbcoh{[X/G]}/\operatorname{Perf}([X/G]).
\]
\end{definition}
\noindent By the following observation of Orlov's,  the category can be viewed as studying the geometry of the singular locus.

\begin{proposition}[Orlov, \cite{O09}]
Assume that $\op{coh}[X/G]$ has enough locally free sheaves. Let $i:U\rightarrow X$ be a $G$-equivariant open immersion such that the singular locus of $X$ is contained in $i(U)$. Then the restriction,
\[
i^\ast:\dsing{[X/G]}\rightarrow\dsing{[U/G]},
\]
is an equivalence of categories.
\end{proposition}

Next, consider a $G$-equivariant vector bundle $\mathcal{E}$ on $X$. Denote by $Z$ the zero locus of a $G$-invariant section $s\in H^0(X,\mathcal{E})$. Then $\langle -,s\rangle$ induces a global function on $\tot\mathcal{E}^\vee$. Let $Y$ be the zero-section of this pairing and consider the fibrewise dilation action on the torus $\mathbb{G}_m$. Then we have the following result.
\begin{Theorem}[Isik \cite{Isi13}, Shipman \cite{Shi12}, Hirano \cite{Hirano}]
\label{Thm:IsikShipman}
Suppose the Koszul complex on $s$ is exact. Then there is an equivalence of categories \[
\dsing{[Y/(G\times\mathbb{G}_m)]}\cong \dbcoh{[Z/G]}.
\]
\end{Theorem}

\noindent Combining the previous two results gives the following.

\begin{Corollary}[Corollary 3.4 in \cite{FK19}]
\label{Cor:FK19-3.4}
Let $V$ be an algebraic variety with a $G\times\mathbb{G}_m$ action. Suppose there is an open subset $U\subseteq V$ such that $U$ is $G\times\mathbb{G}_m$ equivariantly isomorphic to $Y$ as above and that $U$ contains the singular locus of $X$. Then \[
\dsing{[V/(G\times\mathbb{G}_m)]}\cong \dbcoh{[Z/G]}.
\]
\end{Corollary}

We will move towards making these results applicable to the objects studied in this paper, adapting \cite{FK19}.
Consider an affine space $X:=\mathbb{A}^{n+t}$ with coordinates $x_i, u_j$ for $1\le i\le n, 1\le j\le t$. Let $T$ denote the standard open torus $\mathbb{G}_m^{n+t}$ and consider a subgroup $S\subseteq T$, with $\tilde{S}$ the connected component that contains the identity. 

Recall the notion of GKZ fans from $\S \ref{sec:secfans}$. We adjust the notation so that $S$ above corresponds to the group $G$ from $\S \ref{sec:secfans}$. We will now explain how to construct varieties corresponding to the chambers of the GKZ fan, and the goal of this setup is to  apply Corollary $\ref{Cor:FK19-3.4}$ and VGIT to provide equivalences between derived categories.

\begin{definition}
Let $G$ be a group acting on a space $X$ and $f$ a global function on $X$. $f$ is said to be \newterm{semi-invariant} with respect to a character $\chi$ if, for any $g\in G$, $f(g\cdot x)=\chi(g)f(x).$
\end{definition}
To apply Corollary $\ref{Cor:FK19-3.4}$, we will add a $\mathbb{G}_m$-action which is $S$-invariant and $\mathbb{G}_m$-semi-invariant, acting with weight $0$ on the $x_i$ and $1$ on the $u_j$. We refer to this action as \newterm{R-charge}.
Consider the action of $S$ on the scheme $\spec \mathbb{C}[u_j]$. It corresponds to a character $\gamma_j$ of $S$. Let $f_1,\dots,f_t$ be a collection of $S$-semi-invariant functions in the $x_i$ with respect to $\gamma_j^{-1}$. Then define a function, called \newterm{superpotential}, by \[
w:=\sum_{j=1}^t u_j f_j.
\] 
The superpotential $w$ is $S$-invariant and $\chi$-semi-invariant with respect to the projection character $\chi: S\times\mathbb{G}_m\rightarrow\mathbb{G}_m$, hence $w$ is homogeneous of degree 0 with respect to the $S$-action and of degree 1 with respect to the R-charge. Let $Z(w)\subseteq X$ be its zero-locus and define $Y_p:= Z(w)\cap U_p$. Then we have the following result.

\begin{Theorem}[Theorem 3 in \cite{HW12}]
\label{Thm:Herbst-Walcher}
If $S$ is quasi-Calabi-Yau, there is an equivalence of categories\[
\dsing{[Y_p/S\times\mathbb{G}_m]}\cong\dsing{[Y_q/S\times\mathbb{G}_m]}
\]
for all $1\le p,q\le k$, where $k$ is the number of chambers in the GKZ fan.
\end{Theorem}

We will use this result to show a useful equivalence of derived categories. We start by explicitly describing the open sets $U_p$ corresponding to a chamber $\sigma_p$ of the GKZ fan, defined in $\S \ref{sec:secfans}$.
For $1\le p\le k$ we associate an irrelevant ideal $\I_p$ to $\sigma_p$ by considering the (regular) triangulation $\mathcal{T}_p$ that the chamber corresponds to.
So, let 
\[
\I_p:=\left\langle \prod_{i\not\in I}x_i\prod_{j\not\in J}u_j\ \Big\vert\  \bigcup_{i\in I}\nu_i(S)\cup\bigcup_{j\in J}\nu_{n+j}(S)\text{ is the set of vertices of a simplex in }\mathcal{T}_p\right\rangle.
\]
Then $U_p=X\setminus Z(\I_p)$. Another ideal we will need is a subideal of $\I_p$, given similarly to $\I_p$ by requiring $J$ to be the full set $\{1,\dots,t\}$, i.e.,
\[
\J_p:=\left\langle \prod_{i\not\in I}x_i\ \Big\vert\  \bigcup_{i\in I}\nu_i(S)\cup\bigcup_{j=1}^t\nu_{n+j}(S)\text{ is the set of vertices of a simplex in }\mathcal{T}_p\right\rangle.
\]
This ideal is therefore generated by those simplices whose sets of vertices contain all $\nu_{n+j}$ for $1\le j \le t$.
Using this subideal, we get a new open set $V_p:=X\setminus Z(\J_p)\subseteq U_p$.
Since $\J_p$ has no $u_j$ in its generators, we can see it as ideal $\mathcal{J}_p^x$ in $\C[x_1,\dots,x_n]$, giving an open subset of $\mathbb{A}^n$ by $V_p^x:=\mathbb{A}^n\setminus Z(\J_p^x)$. This set gives us a toric stack $X_p:=[V_p^x/S]$.
Now suppose $\J_p$ is non-zero. Then the last two quantities defined are nonempty, and one can show $[V_p/S]$ is a vector bundle over $X_p$, with the inclusion of rings $\C[x_1,\dots,x_n]\rightarrow \C[x_1,\dots,x_n,u_1,\dots,u_t]$ restricting to a $S$-equivariant morphism\[
[V_p/S]\rightarrow[V_p^x/S]=X_p.
\]
This morphism gives the following proposition.
\begin{proposition}[Proposition 4.6 in \cite{FK19}]
\label{Prop:FK19Prop4.6}
Suppose $\J_p$ is non-zero. The morphism $[V_p/S]\rightarrow X_p$ realizes $[V_p/S]$ as the total space of a vector bundle \[
[V_p/S]\cong\tot\bigoplus_{j=1}^t\O(\gamma_j).
\]
Furthermore, the R-charge action of $\mathbb{G}_m$ is the dilation action along fibers. Finally, for each $j$, the function $f_j$ gives a section of $\O(\gamma_j^{-1})$ and the superpotential $w=\sum u_jf_j$ restricts to the pairing with the section $\oplus f_j$.
\end{proposition}

In particular, from this we can view the function $\oplus f_j$ as a section of $V_p$ which defines, for all $p$, a complete intersection $Z_p:=Z(\oplus f_j)\subseteq X_p$. %Restricting the superpotential to the open set $U_p$, we can also consider the zero-locus $Y_p:= Z(w)\cap U_p$.
Finally, we introduce the Jacobian ideal $\partial w$, generated by the partial derivatives of $w$ with respect to the coordinates $x_i,u_j$. 

\begin{proposition}[Proposition 4.7 in \cite{FK19}]
\label{Prop:FK19Prop4.7}
Suppose $\J_p$ is non-zero. If $\I_p\subseteq\sqrt{\partial w,\J_p}$, then\[
\dsing{[Y_p/S\times\mathbb{G}_m]}\cong\dbcoh{Z_p}.
\]
\end{proposition}
\noindent This finally leads us to the following result, which we will use in $\S \ref{sec:linking}$. 
\begin{Corollary}
\label{Cor:FK19Cor4.8}
Assume $S$ satisfies the quasi-Calabi-Yau condition and that $\J_p$ and $\J_q$ are non-zero. If $\I_p\subseteq\sqrt{\partial w,\J_p}$ and $\I_q\subseteq\sqrt{\partial w,\J_q}$ for some $1\le p,q\le r$, then\[
\dbcoh{Z_p}\cong\dbcoh{Z_q}.
\]
\end{Corollary}

\section{The Libgober-Teitelbaum and the Batyrev-Borisov constructions}\label{sec:LT}
\subsection{The Batyrev-Borisov construction in $\P^5$}
\label{sec:BB}

We now construct a Batyrev-Borisov mirror to a complete intersection of two cubics in $\P^5$.
We will do this by giving a nef partition of the anticanonical polytope of $\P^5$ which corresponds to a complete intersection. Then we will apply the Batyrev-Borisov construction to that nef partition, obtaining a polytope $\nabla$ corresponding to the mirror. Fix the lattice $M\cong \Z^5$ and its dual lattice $N$.
\begin{remark}Due to the way we will derive certain fans in this section via methods inspired by mirror symmetry (see $\S$ \ref{rem:LT}) our first fan lives in $M_{\R}$ and not in the conventional $N_{\R}$.
\end{remark}

Define the rays $\overline{\rho_0},\dots,\overline{\rho_{11}}$ in $M_{\R}\oplus \R^2$ with primitive generators

\begin{equation*}\begin{array}{lll}
u_{\overline{\rho_0}}=(3,0,0,-1,-1,0,1), &  & u_{\overline{\rho_6}}=(2,-1,-1,0,0,1,0),\\
u_{\overline{\rho_1}}=(0,3,0,-1,-1,0,1), & & u_{\overline{\rho_7}}=(-1,2,-1,0,0,1,0),\\
u_{\overline{\rho_2}}=(0,0,3,-1,-1,0,1), & & u_{\overline{\rho_8}}=(-1,-1,2,0,0,1,0),\\
u_{\overline{\rho_3}}=(-1,-1,-1,3,0,1,0), & & u_{\overline{\rho_9}}=(0,0,0,2,-1,0,1),\\
u_{\overline{\rho_4}}=(-1,-1,-1,0,3,1,0), & & u_{\overline{\rho_{10}}}=(0,0,0,-1,2,0,1),\\
u_{\overline{\rho_5}}=(-1,-1,-1,0,0,1,0), &  & u_{\overline{\rho_{11}}}=(0,0,0,-1,-1,0,1),\\
u_{\tau_1}=(0,0,0,0,0,1,0), & &  u_{\tau_2}=(0,0,0,0,0,0,1).\;
\end{array}\end{equation*}

\begin{notation}
\label{not:LT}
For $0\le j\le 11$, we denote by $u_{\rho_j}$ the lattice point in $M$ obtained from $u_{\overline{\rho_j}}$ by projecting onto the first 5 coordinates. Denote by $\rho_j$ the ray generated by $u_{\rho_j}$ in $M_{\R}$.
\end{notation}

\begin{proposition}
\label{Prop:fanBBmirror}
Consider the fan $\Sigma_\nabla$ with rays $\rho_0,\dots,\rho_{11}$ defined above and maximal cones listed in Table $~\ref{Tab:irrel}$ (page~\pageref{Tab:irrel}). 
Then a general complete intersection in the toric variety $X_\nabla$ corresponding to the fan $\Sigma_\nabla$ is a Batyrev-Borisov mirror to a complete intersection of two cubics in $\P^5$.
\end{proposition}
\begin{proof}

\noindent The anticanonical sheaf of $\P^5$ is $\O_{\P^5}(6)$, corresponding to the divisor class \[-K_{\P^5}=T_0+\dots+T_5=(T_0+T_1+T_2)+(T_3+T_4+T_5).\] 
 \noindent
The anticanonical polytope for $\P^5$ is given by \[
\Delta_{-K_{\P^5}}=\{m\in M_{\R}\vert\langle m,u_\rho\rangle\geq -1 \text{ for }\rho\in\Sigma_{\P^5}(1)\}\subseteq M_{\R},\] 
 which is the convex hull of the six points \begin{equation*}\begin{array}{lll}
       (5,-1,-1,-1,-1), & (-1,5,-1,-1,-1), & (-1,-1,5,-1,-1),  \\
       (-1,-1,-1,5,-1), & (-1,-1,-1,-1,5), & (-1,-1,-1,-1,-1).
 \end{array}
 \end{equation*}
A nef partition with respect to the origin of the polytope $\Delta_{-K_{\P^5}}$ is given by the polytopes $\Delta_1, \Delta_2$ associated to the divisors $T_0+T_1+T_2$ and $T_3+T_4+T_5$, since the Minkowski sum $\Delta_1+\Delta_2$ is equal to $\Delta_{-K_{\P^5}}$.
These polytopes are 
\begin{equation*}\begin{array}{ll}
\Delta_1=& \operatorname{Conv}((2,-1,-1,0,0),(-1,2,-1,0,0),(-1,-1,2,0,0),\\
& (-1,-1,-1,3,0),(-1,-1,-1,0,3),(-1,-1,-1,0,0)),\\
\Delta_2=&\operatorname{Conv}((0,0,0,-1,2),(0,0,0,2,-1),(0,0,3,-1,-1),\\
&(0,3,0,-1,-1),(3,0,0,-1,-1),(0,0,0,-1,-1)).
\end{array}\end{equation*}

\noindent
Next, we shall compute the dual nef partition, as defined in $\S \ref{sec:Back}$. We have:
\begin{equation*}\begin{array}{ll}
\nabla_1=&\operatorname{Conv}((1,0,0,0,0),(0,1,0,0,0),(0,0,1,0,0),(0,0,0,0,0))\\
\nabla_2=&\operatorname{Conv}((0,0,0,0,1),(0,0,0,1,0),(0,0,0,0,0),(-1,-1,-1,-1,-1)).
\end{array}\end{equation*}

\noindent Their Minkowski sum $\nabla\subseteq N_{\R}$ is then the convex hull of the 15 points 
\begin{equation*}\begin{array}{llll}
(1,0,0,0,0), & (0,1,0,0,0), & (0,0,1,0,0), & (0,0,0,1,0), \\
(0,0,0,0,1), & (1,0,0,1,0), & (1,0,0,0,1), & (0,1,0,1,0), \\
(0,1,0,0,1), & (0,0,1,1,0), & (0,0,1,0,1), & (-1,-1,-1,-1,-1), \\
(0,-1,-1,-1,-1), & (-1,0,-1,-1,-1), & (-1,-1,0,-1,-1). & \;
\end{array}\end{equation*}

\noindent
A SAGE computation shows the normal fan of $\nabla$, $\Sigma_\nabla'\subseteq M_{\R}$, has rays $\rho_0,\dots,\rho_{11}$ from Notation \ref{not:LT}.
The maximal-dimensional cones are the following 15 cones:
\begin{equation*}\begin{array}{lllllll}
\rho_0\rho_1\rho_2\rho_9\rho_{10}, & &
\rho_0\rho_1\rho_3\rho_4\rho_6\rho_7\rho_9 
\rho_{10}, & &
\rho_0\rho_2\rho_3\rho_4\rho_6\rho_8\rho_9
\rho_{10}, & &
\rho_1\rho_2\rho_3\rho_4\rho_7\rho_8\rho_9 \rho_{10},\\
\rho_1\rho_2\rho_4\rho_5\rho_7\rho_8\rho_{10}
\rho_{11}, & &
\rho_0\rho_1\rho_2\rho_9\rho_{11}, & &
\rho_0\rho_1\rho_2\rho_{10}\rho_{11}, & &
\rho_3\rho_4\rho_5\rho_6\rho_7,\\
\rho_0\rho_1\rho_3\rho_5\rho_6\rho_7\rho_9\rho_{11}, & &
 \rho_0\rho_1\rho_2\rho_5\rho_6\rho_7\rho_{10}\rho_{11},& &
\rho_3\rho_4\rho_5\rho_6\rho_8, & &
\rho_0\rho_2\rho_3\rho_5\rho_6\rho_8\rho_9\rho_{11},\\
\rho_0\rho_2\rho_4\rho_5\rho_6\rho_8\rho_{10}\rho_{11}, & &
\rho_3\rho_4\rho_5\rho_7\rho_8, & &
\rho_1\rho_2\rho_3\rho_5\rho_7\rho_8\rho_9\rho_{10}\rho_{11}. & \;
\end{array}\end{equation*}

\noindent We listed the cones by giving the rays generating them. For instance, $\rho_0\rho_1\rho_2\rho_9\rho_{10}$ stands for the cone $\cone{(\rho_0,\rho_1,\rho_2,\rho_9,\rho_{10})}$.
Note here that some of these maximal cones contain more rays than the others. So, as described in Remark \ref{Rem:BB}, we want a MPCP-resolution of the variety associated to the above fan. To do this, we subdivide each of the maximal cones which has more than 5 rays. This procedure involves choice, as each cone can be subdivided in 24 ways (being a total of $24^9$ possible choices!). However, all these choices are related by GIT, so any choice gives us a mirror family, all of which are birational.
\noindent
Following this procedure, the Table $\ref{Tab:irrel}$ (see below) gives the 42 maximal cones in the fan corresponding to a MPCP-resolution of the variety associated to the fan $\Sigma_\nabla'$. Define the fan $\Sigma_\nabla$ to be the fan consisting of those 42 5-dimensional cones and all of their faces.
Determining the variety $X_\nabla$ explicitly is not straightforward, but also not necessary for our purposes, so long as we have the fan $\Sigma_\nabla$.
\end{proof}

\begin{table}[htb!]
\[\begin{array}{|r|r|r|r|r|r|}
\hline
\rho_0\rho_1\rho_2\rho_9\rho_{10} & 
\rho_0\rho_1\rho_6\rho_9\rho_{10} &
\rho_3\rho_4\rho_6\rho_7\rho_9 & 
\rho_1\rho_6\rho_7\rho_9\rho_{10} &
\rho_4\rho_6\rho_7\rho_9\rho_{10} & 
\rho_0\rho_2\rho_6\rho_9\rho_{10} \\
\hline
\rho_2\rho_6\rho_8\rho_9\rho_{10} & 
\rho_3\rho_4\rho_6\rho_8\rho_9  &
\rho_4\rho_6\rho_8\rho_9\rho_{10} & 
\rho_1\rho_2\rho_7\rho_9\rho_{10} &
\rho_2\rho_7\rho_8\rho_9\rho_{10} & 
\rho_3\rho_4\rho_7\rho_8\rho_{10}  \\
\hline
\rho_5\rho_7\rho_8\rho_9\rho_{10} & 
\rho_1\rho_2\rho_7\rho_{10}\rho_{11} &
\rho_2\rho_7\rho_8\rho_{10}\rho_{11} & 
\rho_4\rho_5\rho_7\rho_8\rho_{10} &
\rho_5\rho_7\rho_8\rho_{10}\rho_{11} & 
\rho_0\rho_1\rho_2\rho_9\rho_{11}  \\
\hline
\rho_0\rho_1\rho_2\rho_{10}\rho_{11} & 
\rho_3\rho_4\rho_5\rho_6\rho_7 & 
\rho_0\rho_1\rho_6\rho_9\rho_{11} & 
\rho_1\rho_6\rho_7\rho_9\rho_{11}  &
\rho_3\rho_5\rho_6\rho_7\rho_9 & 
\rho_5\rho_6\rho_7\rho_9\rho_{11} \\
\hline
\rho_0\rho_1\rho_6\rho_{10}\rho_{11} & 
\rho_1\rho_6\rho_7\rho_{10}\rho_{11}  &
\rho_4\rho_5\rho_6\rho_7\rho_{10} &
\rho_5\rho_6\rho_7\rho_{10}\rho_{11}  &
\rho_3\rho_4\rho_5\rho_6\rho_8 &
\rho_0\rho_2\rho_6\rho_9\rho_{11}  \\
\hline
\rho_2\rho_6\rho_8\rho_9\rho_{11} & 
\rho_3\rho_5\rho_6\rho_8\rho_9  &
\rho_5\rho_6\rho_8\rho_9\rho_{11} & 
\rho_0\rho_2\rho_6\rho_{10}\rho_{11}  &
\rho_2\rho_6\rho_8\rho_{10}\rho_{11} & 
\rho_4\rho_5\rho_6\rho_8\rho_{10}  \\
\hline
\rho_5\rho_6\rho_8\rho_{10}\rho_{11} & 
\rho_3\rho_4\rho_5\rho_7\rho_8  &
\rho_1\rho_2\rho_7\rho_9\rho_{11} & 
\rho_2\rho_7\rho_8\rho_9\rho_{11}  &
\rho_3\rho_5\rho_7\rho_8\rho_9 & 
\rho_5\rho_7\rho_8\rho_9\rho_{11}  \\
\hline
\end{array}\]
\caption{Maximal cones of $X_\nabla$}
\label{Tab:irrel}
\end{table}

For $i=0,\dots,11$, call $D_i'$ the torus-invariant divisor on $X_\nabla$ corresponding to the ray $\rho_i$ of $\Sigma_\Delta$. 
Let $D_a'=D_0'+D_1'+D_2'+D_9'+D_{10}'+D_{11}'$ and $D_b'=D_3'+D_4'+D_5'+D_6'+D_7'+D_8'$. 

\begin{Corollary}\label{Cor:TotXnabla}
Let $\Sigma_{\nabla,D_a',D_b'}$ be the fan with rays $\overline{\rho_0},\dots,\overline{\rho_{11}},\tau_1,\tau_2$, and cones over those rays inherited from $\Sigma_\nabla$. Then $\Sigma_{\nabla,D_a',D_b'}$ is a fan corresponding to $\tot(\O_{X_\nabla}(-D_b')\oplus\O_{X_\nabla}(-D_a'))$.
\end{Corollary}
\begin{proof}
Apply Proposition \ref{prop:7.3.1CLS} twice to get the result (recalling that we can do this by Proposition \ref{Prop:repeatProp7.3.1}).
\end{proof}

\subsection{Libgober and Teitelbaum's Mirror}\label{sec:LTintro}
We now recall the  family Libgober and Teitelbaum give as a mirror to the generic complete intersection of 
two cubics in $\P^5$. To start, define $V_\lambda\subseteq \P^5$ to be the vanishing set of the following two polynomials:
\begin{equation}
\label{eq:DefQi}
Q_{1,\lambda}=x_0^3+x_1^3+x_2^3-3\lambda x_3x_4x_5,\quad \quad
Q_{2,\lambda}=x_3^3+x_4^3+x_5^3-3\lambda x_0x_1x_2.
\end{equation}

For generic $\lambda$, this gives a smooth complete intersection in $\P^5$ which is a Calabi-Yau threefold.

Let $\zeta_n$ denote a primitive $n$-th root of unity. Let $\alpha,\beta,\delta,\epsilon\in\Z\pmod 3$ and $\mu\in\Z\pmod 9$ with %$\zeta_9^{3\mu}=\zeta_3^{\alpha+\beta}=\zeta_3^{\delta+\epsilon}$.
$3\mu=\alpha+\beta=\delta+\epsilon$. Define the diagonal matrix \begin{eqnarray*}
&& g_{\alpha,\beta,\delta,\epsilon,\mu}:= \operatorname{diag}\left( \zeta_3^\alpha\zeta_9^\mu,\,\, \zeta_3^\beta\zeta_9^\mu,\,\, \zeta_9^\mu,\,\, \zeta_3^{-\delta}\zeta_9^{-\mu},\,\, \zeta_3^{-\epsilon}\zeta_9^{-\mu},\,\, \zeta_9^{-\mu}\right)\;
\end{eqnarray*}
\noindent and let $G_{81}\subset PGL(5,\C)$ denote the order 81 group generated by the $g_{\alpha,\beta,\delta,\epsilon,\mu}$. Note that $G_{81}$ acts on $\P^5$ by restricting the natural action of $PGL(5,\C)$ on $\P^5$. The polynomials $Q_{1,\lambda}, Q_{2,\lambda}$ are invariant with respect to the action of $G_{81}$, hence $G_{81}$ acts on $V_{\lambda}$.

Note that $G_{81}$ is of isomorphism type $(\Z/3\Z)^2\times (\Z/9\Z)$ and can be generated by $(\zeta_3,\zeta_3^{-1},1,1,1,1)$, $(1,1,1,\zeta_3^{-1},\zeta_3,1)$ and $(\zeta_9,\zeta_9^4,\zeta_9,\zeta_9^{-1},\zeta_9^{-4},\zeta_9^{-1})$.

 Let $V_{LT,\lambda}$ be the quotient of $V_\lambda$ by the action of $G_{81}$ and let $W_{LT,\lambda}$ be a minimal resolution of singularities of $V_{LT,\lambda}$ which is a Calabi-Yau manifold. 

\subsection{Expressing Libgober-Teitelbaum torically}\label{sec:LTexp}
\noindent In the following, we aim to give a toric description of $V_{LT,\lambda}$. First we give a fan for the toric variety $X_{LT}:=\P^5/G_{81}$ and then employ methods of $\S 7.3$ of \cite{CLS} to construct a vector bundle over $X_{LT}$ that has the global section $Q_{1,\lambda}\oplus Q_{2,\lambda}$. 

%\subsubsection{The variety $X_{LT}$} 

\begin{proposition}
\label{Prop:FanforP5G81} 
Consider the 1-dimensional cones $\rho_0,\dots,\rho_5$ with corresponding primitive generators

\begin{equation*}\begin{array}{llllll}
     u_{\rho_0}&=(3,0,0,-1,-1), & u_{\rho_1}&=(0,3,0,-1,-1), & u_{\rho_2}&=(0,0,3,-1,-1), \\
    u_{\rho_3}&=(-1,-1,-1,3,0), &
    u_{\rho_4}&=(-1,-1,-1,0,3), &
    u_{\rho_5}&=(-1,-1,-1,0,0). \;
\end{array}\end{equation*}

\noindent Consider the collection $\mathcal{C}$ of sets of the form \[
\{\rho_i \,\, \vert \,\,  i\in I, I\subseteq\{0,\dots,5\}, |I|=5\}.
\]
Let $\Sigma_{LT}\subseteq M_{\R}$ be the fan consisting of maximal cones \[\{\cone(C)\vert C\in\mathcal{C}\}\] and all their faces.
    
Then the toric stack associated to $\Sigma_{LT}$ is the stack corresponding to the Libgober-Teitelbaum construction, $\mathcal{X}_{LT}=[\C^6\setminus\{0\}/\left(\C^\ast\times G_{81}\right)]$, with $\C^\ast$ acting by $(\lambda x_0,\dots,\lambda x_5)\sim (x_0,\dots,x_5)$ and $G_{81}$ acting as described above in $\S$ \ref{sec:LTintro}.
\end{proposition}
\begin{proof}

  We use the Cox construction described in $\S \ref{coxconstruction}$. 
 By Lemma $\ref{Lem:5.1.1CLS}$, we obtain the following system of equations characterising elements of $G:= G_\Sigma$
\begin{eqnarray}
&& t_3t_4t_5=t_0^3\label{eqn:I}\\
&& t_3t_4t_5=t_1^3\label{eqn:II}\\
&& t_3t_4t_5=t_2^3\label{eqn:III}\\
&& t_0t_1t_2=t_3^3\label{eqn:IV}\\
&& t_0t_1t_2=t_4^3\label{eqn:V}\;
\end{eqnarray}

First we note that we have a copy of $\C^\ast$ in $G$, given by $\{t\cdot(1,1,1,1,1,1)\ \vert\ t\in\C^\ast\}$, so to compute $G$ we consider the group $H$ of cosets of $\C^\ast$. We will explicitly describe $H$ and subsequently use the direct product theorem to compute $G$.
Consider an element $(t_0,\dots,t_5)\in G$. By an appropriate choice of coset representative of $(t_0,t_1,t_2,t_3,t_4,t_5)\cdot \C^\ast$, we may assume $\prod_{i=0}^5 t_i=1$.

Using equations \eqref{eqn:I}, \eqref{eqn:II} and \eqref{eqn:III}, we have $t_0^3=t_1^3=t_2^3$, and thus $t_0=\zeta_3^\alpha t_2, t_1=\zeta_3^\beta t_2$ for some $\alpha,\beta\in\Z_3$.
Using equations $(\ref{eqn:I})-(\ref{eqn:V})$, we have that $t_3^3t_4^3t_5^3=t_0^3t_1^3t_2^3=t_3^6t_4^3=t_3^3t_4^6$, which implies \begin{equation}\label{eqn:VI}t_5^3=t_3^3=t_4^3.\end{equation} Hence, similarly to above, we obtain $t_3=\zeta_3^{-\delta}t_5, t_4=\zeta_3^{-\varepsilon}t_5$ for some $\delta,\varepsilon\in\Z_3$.

\noindent By combining $(\ref{eqn:III})$, $(\ref{eqn:IV})$ and $(\ref{eqn:VI})$ we obtain \begin{equation}\label{eqn:t3inverse} t_2^3t_5^3=t_0t_1t_2t_3t_4t_5=1.
\end{equation} 
Equation \eqref{eqn:t3inverse} implies $t_5^3=(t_2^{-1})^3$, thus $t_5=\zeta_3^{\nu}\cdot t_2^{-1}$ for some $\nu\in\Z_3$. Using $t_3=\zeta_3^{-\delta}t_5$ and $t_4=\zeta_3^{-\varepsilon}t_5$ and equation $(\ref{eqn:III})$, we obtain
$$
t_2^3=t_3t_4t_5 = t_5^3\zeta_3^{-(\delta+\varepsilon)}=t_2^{-3}\zeta_3^{-(\delta+\varepsilon)}.
$$
Hence $t_2^{18}=1$. So we can write $t_2=\zeta_{18}^l$ for some $l\in\Z_{18}$. 

 We now claim that $t_2$ can be assumed to be a ninth root of unity and $t_5$ to be its inverse, i.e. $t_2=\zeta_9^\mu, t_5=\zeta_9^{-\mu}$ for some $\mu\in\Z_9$. Indeed, note that $(\zeta_6,\dots,\zeta_6)\in (1,1,1,1,1,1)\cdot\C^\ast\subseteq G$, so we can scale an element $(t_0,\dots,t_5)\in G$ by sixth roots of unity, leaving the product $\prod_{i=1}^6 t_i$ invariant. The claim follows by multiplication with an appropriate sixth root of unity.

 Expressing all the $t_i$ in terms of $t_2$, the assumption $1=\prod_{i=1}^6t_i$ implies $1=\zeta_3^{\alpha+\beta-\delta-\varepsilon}$, or, equivalently, \[\alpha+\beta=\delta+\varepsilon\pmod 3.\] Finally, using $(\ref{eqn:III})$ gives $\zeta_9^{3\mu}=\zeta_3^{-\delta+\varepsilon}\zeta_9^{-3\mu}$ and therefore $\zeta_9^{3\mu}=\zeta_3^{\delta+\varepsilon}$. Thus $H$, the group of cosets of $\C^\ast$, is isomorphic to $G_{81}$, where $G_{81}$ is the same group described in $\S\ref{sec:Back}$. In particular, all elements of $G$ are of the form $g\cdot \lambda$ with $g\in G_{81},\ \lambda\in (1,1,1,1,1,1)\cdot\C^\ast$ and $G_{81}\cap\{(1,1,1,1,1,1)\cdot\lambda\vert\lambda\in\C^\ast\}=\{(1,1,1,1,1,1)\}$.
 Hence, by the direct product theorem, $G\cong \C^\ast\times G_{81}$.

The Cox fan of $\Sigma_{LT}$ can be described as follows. It has six rays $e_{\rho_0},\dots,e_{\rho_5}$. It is straightforward to see that the maximal cones are all 5-dimensional cones generated by any 5 of the rays above. 
Therefore, we obtain $U_{\Sigma_{LT}}=\mathbb{A}^6\setminus\{0\}$. Thus, the Cox stack associated to $\Sigma_{LT}$ is 
\[
\mathcal{X}_{LT}=[U_{\Sigma_{LT}}/ G]=[\C^6\setminus\{0\}/\left(\C^\ast\times G_{81}\right)],
\]
with the prescribed action, as required.
\end{proof}

\begin{remark} We note that by Theorem $\ref{Thm:FK18Stacks}$ the coarse moduli space of the stack $\mathcal{X}_{LT}$ is $X_{LT}$, since $\Sigma_{LT}$ is simplicial.
\end{remark}

Starting with the fan $\Sigma_{LT}$ of $X_{LT}$, we apply Proposition \ref{prop:7.3.1CLS} twice to construct a vector bundle. Let $D_i$ be the Weil divisor corresponding to the ray $\rho_i$ in $\Sigma_{LT}$. Let $D_a=D_0+D_1+D_2$ and $D_b=D_3+D_4+D_5$.

\begin{Corollary}
\label{Lem:totalspaceLT}
Denote by the rays $\overline{\rho_0},\dots,\overline{\rho_5}$, $\tau_1$ and $\tau_2$ the rays\footnote{These are the same as on page~\pageref{not:LT}.} generated by the primitive generators:
\begin{equation*}\begin{array}{lll}
u_{\overline{\rho_0}}=(3,0,0,-1,-1,0,1),     & u_{\overline{\rho_1}}=(0,3,0,-1,-1,0,1),  &
u_{\overline{\rho_2}}=(0,0,3,-1,-1,0,1),    \\ u_{\overline{\rho_3}}=(-1,-1,-1,3,0,1,0), &
u_{\overline{\rho_4}}=(-1,-1,-1,0,3,1,0),  & u_{\overline{\rho_5}}=(-1,-1,-1,0,0,1,0), \\
u_{\tau_1}=(0,0,0,0,0,1,0), & u_{\tau_2}=(0,0,0,0,0,0,1). & \;
\end{array}\end{equation*}
Consider the collection $\mathcal{S}$ of sets of the form \[
\{\overline{\rho_i} \,\, \vert \,\,  i\in I, I\subseteq\{0,\dots,5\}, |I|=5\}\cup\{\tau_1,\tau_2\}.
\]
Let $\Sigma_{LT,D_a,D_b}$ be the fan in $M_{\R}\oplus \R^2$ consisting of the maximal cones \[\{\cone(S)\vert S\in\mathcal{S}\}\] and all their faces. Then:
 \begin{enumerate}[(a)]
     \item $\Sigma_{LT,D_a,D_b}$ is a fan corresponding to $\tot(\O_{X_{LT}}(-D_b)\oplus\O_{X_{LT}}(-D_a))$;
    \item The vector bundle $\O_{X_{LT}}(D_b)\oplus\O_{X_{LT}}(D_a)$ has the global section $Q_{1,\lambda}\oplus Q_{2,\lambda}$.
 \end{enumerate}
\end{Corollary}

\begin{proof}  Applying Proposition $\ref{prop:7.3.1CLS}$ twice yields (a). 

We now turn to (b) and show that $Q_{1,\lambda}\in\Gamma(X_{LT},\O_{X_{LT}}(D_b))$ and $ Q_{2,\lambda}\in\Gamma(X_{LT},\O_{X_{LT}}(D_a))$. 
We start by noting that on $X_{LT}$ we have $\div(x_i^3)=3D_i$, so $\div(x_i^3)-3D_i\geq 0$, i.e. $x_i^3\in\Gamma(X_{LT},\O_{X_{LT}}(3D_i))$. Similarly, $x_0x_1x_2\in\Gamma(X_{LT},\O_{X_{LT}}(D_a))$ and $x_3x_4x_5\in\Gamma(X_{LT},\O_{X_{LT}}(D_b))$.

To show the linear equivalence of two divisors, it suffices to consider their difference and show it is principal. We recall that $\div(\chi^n)=\sum_{\rho\in\Sigma(1)} \langle u_\rho,n\rangle D_\rho$, corresponding to the map $\iota$ in the exact sequence \eqref{eq:Divisor}. So, for instance $3D_1-3D_0=\div(x_0^{-3}x_1^3)$ which is the character associated to the lattice point $(-1,1,0,0,0)$. Hence $3D_1-3D_0=0$ in $\operatorname{Cl}(X_{LT})$, i.e. $3D_0\sim 3D_1$. Similarly $3D_1\sim 3D_2$ and $3D_3\sim3D_4\sim3D_5$. Using the lattice points $(-1,0,0,0,0)$ and $(0,0,0,-1,0)$ respectively, we also see that $3D_0\sim D_b$ and $3D_3\sim D_a$.

Thus \[\O_{X_{LT}}(3D_0)\simeq\O_{X_{LT}}(3D_1)\simeq\O_{X_{LT}}(3D_2)\simeq\O_{X_{LT}}(D_b)\] and 
\[\O_{X_{LT}}(3D_3)\simeq\O_{X_{LT}}(3D_4)\simeq\O_{X_{LT}}(3D_5)\simeq\O_{X_{LT}}(D_a),\] implying $Q_{2,\lambda}\in\Gamma(X_{LT},\O_{X_{LT}}(D_a))$ and $Q_{1,\lambda}\in\Gamma(X_{LT},\O_{X_{LT}}(D_b))$, as required.
\end{proof}

\subsubsection{Intuition for constructing $X_{LT}$ torically}
\label{rem:LT}

We now explain how we found an explicit description for the fan $\Sigma_{LT}$.
We start by considering the standard fan $\Sigma_{\P^5}\subseteq N_{\R}$ for $\P^5$ in the standard basis. It is the fan consisting of the cones generated by any proper subset of the six rays $\nu_0,\dots,\nu_5 $ with primitive generators
\begin{equation*}\begin{array}{lllll}
u_{\nu_0}=(1,0,0,0,0), & & u_{\nu_1}=(0,1,0,0,0),& & u_{\nu_2}=(0,0,1,0,0),\\  u_{\nu_3}=(0,0,0,1,0),& & u_{\nu_4}=(0,0,0,0,1), & & u_{\nu_5}=(-1,-1,-1,-1,-1).\;
\end{array}\end{equation*}

\noindent
Denote by $T_0,\dots,T_5$ the six primitive Weil divisors corresponding to the rays $u_{\nu_0},\dots,u_{\nu_5}$ respectively. Then \[\O(-\underbrace{(T_0+T_1+T_2)}_{:=T_a})=\O(-\underbrace{(T_3+T_4+T_5)}_{:=T_b})=\O(-3),\] and we can use the methods of $\S 7.3$ of \cite{CLS} again to construct a fan of $\tot(\O_{\P^5}(-3)\oplus\O_{\P^5}(-3))$. This yields the fan $\Sigma_{\P^5,T_a,T_b}$ in $N_{\R}\oplus\R^2$ with the 8 rays $\overline{\nu_0},\dots,\overline{\nu_5},\tau_1$ and $\tau_2$ having primitive ray generators 
\begin{equation}\label{the nus and taus}\begin{array}{lll}
u_{\overline{\nu_0}}=(1,0,0,0,0,1,0), & & u_{\overline{\nu_4}}=(0,0,0,0,1,0,1),\\
u_{\overline{\nu_1}}=(0,1,0,0,0,1,0), & &
u_{\overline{\nu_5}}=(-1,-1,-1,-1,-1,0,1),\\ u_{\overline{\nu_2}}=(0,0,1,0,0,1,0), & & u_{\tau_1}=(0,0,0,0,0,1,0),\\ u_{\overline{\nu_3}}=(0,0,0,1,0,0,1), & & 
u_{\tau_2}=(0,0,0,0,0,0,1).
\end{array}\end{equation}

\noindent The fan $\Sigma_{\P^5,T_a,T_b}$ is the star subdivision of $\cone(u_{\overline{\nu_0}},\dots,u_{\overline{\nu_5}},u_{\tau_1},u_{\tau_2})$ along $u_{\tau_1}$ and $u_{\tau_2}$ (noting the abuse of notation by which $u_{\tau_i}$ represent the same vector in both lattices $M,N$). The dual cone to $\Sigma_{\P^5,T_a,T_b}$ in $M_{\R}\oplus\R^2$ is spanned by the 12 rays $\overline{\rho_0},\dots,\overline{\rho_{11}}$ defined in $\S$ \ref{sec:BB} (page~\pageref{sec:BB}).

We recall that each lattice point in the interior of the dual cone corresponds to a global function of $X_{\Sigma_{\P^5},T_a,T_b}$ by associating $m$ to the monomial \[x^m:=\prod_{\rho\in\Sigma_{{\P}^5,T_1,T_2}(1)}x_\rho^{\langle m,u_\rho\rangle}.\]

Now a section $s_1\oplus s_2\in\Gamma(\P^5,\O(3)\oplus\O(3))$ will correspond to a global function on $\tot\left(\O(-3)\oplus\O(-3)\right)$ of the form $u_1s_1+u_2s_2$, where $u_i$ is the variable corresponding to $u_{\tau_i}$. Recalling the polynomials $Q_i$ from $\eqref{eq:DefQi}$ in $\S\ref{sec:LTintro}$, we would like to express the global function $F:=u_2Q_{1,\lambda}+u_1Q_{2,\lambda}$ as a linear combination of global functions of the form $x^m$. We do this by finding the lattice points in the dual cone corresponding to each monomial in $F$.

By splitting it up into its monomials, $u_2Q_{1,\lambda}$ corresponds to the 4 points $(3,0,0,-1,-1,0,1)$, $(0,3,0,-1,-1,0,1)$, $(0,0,3,-1,-1,0,1)$ and $(0,0,0,0,0,0,1)$.

\noindent Similarly, $u_1Q_{2,\lambda}$ corresponds to the points $(-1,-1,-1,3,0,1,0),$  $(-1,-1,-1,0,3,1,0)$, $(-1,-1,-1,0,0,1,0)$ and $(0,0,0,0,0,1,0)$.

We find that these 8 points are the primitive generators for the rays of $\Sigma_{LT,D_a,D_b}$ (see Corollary \ref{Lem:totalspaceLT}).

Quotienting $M_{\R}\oplus\R^2$ by the rays associated to the bundle coordinates (i.e. the lattice points that are the elements of the dual basis dual to $u_{\tau_1}$ and $u_{\tau_2}$) corresponds to a toric morphism $X_{\Sigma_{LT,D_a,D_b}}\rightarrow X_{\Sigma_{LT}}$. 
We emphasize that the dual cone to $\cone(\Sigma_{\P^5,T_a,T_b}(1))$ is given by
 $\operatorname{Conv}(u_{\overline{\rho_0}},\dots,u_{\overline{\rho_{11}}})$. Here, we take a subcone generated by a subset of $\{u_{\overline{\rho_0}},\dots,u_{\overline{\rho_{11}}}\}$.

\subsubsection{Expressing the zerolocus of $Q_{1,\lambda},Q_{2,\lambda}$}\label{sec:zerolocusLT}

We remark that the cone $|\Sigma_{LT, D_a, D_b}|$ is not a reflexive Gorenstein cone, hence the Batyrev-Borisov construction does not apply to it. 

The variety $V_{LT,\lambda}\subseteq X_{LT}$ is the zero-locus of the polynomials $Q_{1,\lambda},Q_{2,\lambda}$, where $Q_{1,\lambda}\oplus Q_{2,\lambda}$ is a section of the vector bundle constructed above in Corollary $\ref{Lem:totalspaceLT}$. 
 Proceeding in the same way as in \textsection\ref{rem:LT}, we consider lattice points on the cone $|\Sigma_{LT,D_a,D_b}|^\vee\subseteq N_{\R}\oplus\R^2$ to get global functions of $\mathcal{X}_{\Sigma_{LT,D_a,D_b}}$. The cone $|\Sigma_{LT,D_a,D_b}|^\vee$ is the cone over the convex hull of the following 12 points:
\begin{equation*}\begin{array}{llll}
  (1,0,0,0,0,1,0), &
 ( 0,  1,  0,  0,  0, 1, 0), & ( 0,  0,  1,  0,  0, 1, 0),\\  ( 0,  0,  0,  1,  0, 0, 1), & ( 0,  0,  0,  0,  1, 0, 1)  &
( 2, -1, -1,  0,  0, 0, 3),\\ (-1,  2, -1,  0,  0, 0, 3), & (-1, -1,  2,  0,  0, 0, 3), & (1,  1,  1,  3,  0, 3, 0),\\  ( 1,  1,  1,  0,  3, 3, 0), &(-1, -1, -1, -1, -1, 0, 1), & (-2, -2, -2, -3, -3, 3, 0).\;
\end{array}\end{equation*}

The points corresponding to the monomials in $u_1Q_{1,\lambda}+u_2Q_{2,\lambda}$, and hence to the section $Q_{1,\lambda}\oplus Q_{2,\lambda}$, are  the lattice points $u_{\overline{\nu_i}}$ and $u_{\tau_i}$ in~\eqref{the nus and taus}.
Later on, describing $V_{LT}$ by these 8 points will allow us to work with $\dbcoh{V_{LT}}$, using results in \cite{FK19}.

\begin{remark}
In their recent work \cite{Rossi20,Rossi21}, Rossi proposes a generalisation of the Batyrev-Borisov mirror construction, called \newterm{framed duality}  (\newterm{f-duality}). $f$-duality provides gives an algorithm to obtain mirror candidates of hypersurfaces and complete intersections in toric varieties. 
Applying $f$-duality to $V_{LT}\subset\P^5/G_{81}$ produces $V_\lambda\subset \P^5$, which in turn gives the same mirror as the Batyrev-Borisov construction when applying $f$-duality to it. Theorem \ref{Thm:MyThmInIntro} suggests that different mirror candidates obtained via $f$-duality may be derived equivalent and prompts the question under what conditions this is the case.
\end{remark}

\section{A derived equivalence between the constructions by Libgober-Teitelbaum and Batyrev-Borisov}\label{sec:linking}

Here we will prove the main result, Theorem \ref{Thm:MyThmInIntro}.

\subsection{Picking a partial compactification}
Looking at the dual of the fan $\Sigma_{LT,D_a,D_b}$ as in Corollary \ref{Lem:totalspaceLT}, we recall from $\S\ref{sec:zerolocusLT}$ that the global function $u_1Q_{1,\lambda}+u_2Q_{2,\lambda}$ corresponds to the points

\begin{equation*}\begin{array}{lll}( 1,  0,  0,  0,  0, 1, 0),& ( 0,  1,  0,  0,  0, 1, 0), & ( 0,  0,  1,  0,  0, 1, 0)\\ ( 0,  0,  0,  1,  0, 0, 1),& ( 0,  0,  0,  0,  1, 0, 1),& (-1, -1, -1, -1, -1, 0, 1),\\  ( 0,  0,  0,  0,  0, 0, 1),& ( 0,  0,  0,  0,  0, 1, 0),& \\\end{array}\end{equation*}

Consider the GKZ fan of $\tot(\O_{X_\nabla}(-D_b')\oplus\O_{X_\nabla}(-D_a'))$. We note that the chambers of this GKZ fan correspond to regular triangulations of the polytope $\mathfrak{P}=\operatorname{Conv}(\mathfrak{C})$, where $\mathfrak{C}$ is the collection of the following 14 points:
\[
\begin{array}{llll}
     P_0=& (3,0,0,-1,-1,0,1),  & P_{6}=&(2,-1,-1,0,0,1,0), \\
     P_1=& (0,3,0,-1,-1,0,1) & P_7=&(-1,2,-1,0,0,1,0) \\
     P_2=& (0,0,3,-1,-1,0,1), & P_{8}=&(-1,-1,2,0,0,1,0), \\
     P_{3}=& (-1,-1,-1,3,0,1,0), & P_{9}= & (0,0,0,2,-1,0,1), \\
     P_{4}=&(-1,-1,-1,0,3,1,0), & P_{10}=&(0,0,0,-1,2,0,1), \\
     P_{5}=&(-1,-1,-1,0,0,1,0), & P_{11}=& (0,0,0,-1,-1,0,1), \\
     S_1=& (0,0,0,0,0,1,0), & S_2=& (0,0,0,0,0,0,1).
\end{array}
\]

In the (regular) triangulations of $\mathfrak{P}$, we look for a subtriangulation corresponding to $\Sigma_{LT,D_a,D_b}$, as then we obtain a partial compactification of $\tot(\O_{X_{LT}}(-D_b)\oplus\O_{X_{LT}}(-D_a))$ from Corollary \ref{Lem:totalspaceLT}.

\begin{proposition}\label{Prop:LTpartiallycpt}

There exists a chamber $\sigma_{LT}$ in the GKZ fan of $\tot(\O_{X_\nabla}(-D_b')\oplus\O_{X_\nabla}(-D_a'))$ (from Corollary \ref{Cor:TotXnabla}) so that the triangulation $\mathcal{T}$ corresponding to the chamber $\sigma_{LT}$ (in the sense of $\ref{Thm:chambers2nd}$) has the following properties:

\begin{itemize}
    \item $\mathcal{T}$ contains the following set of simplices, listed via their vertices: \[\mathcal{T}_0:=\left\{\{P_i,S_1,S_2\ \vert \  i\in I\}\ \vert\ I \subset \{0,2,\dots,5\}, |I|=5\right\}.\] 
    \item Any simplex $T\in\mathcal{T}\setminus\mathcal{T}_0$ fulfills either of the two following conditions:
    \begin{enumerate}
        \item $S_1,P_6,P_7,P_8\not\in T$ and $\exists\ 3\le j\le 6$ such that $P_j,P_{6+j}\not\in T$.
        \item $S_2,P_9,P_{10},P_{11}\not\in T$ and $\exists\ 0\le j\le 2$ such that $P_j,P_{6+j}\not\in T$.
    \end{enumerate}
\end{itemize}
Moreover, the toric variety $X_\Sigma$ corresponding to the chamber $\sigma_{LT}$ is a partial compactification of the variety $\tot(\O_{X_{LT}}(-D_b)\oplus\O_{X_{LT}}(-D_a))$ from Corollary \ref{Lem:totalspaceLT}.
\end{proposition}

The first property of $\mathcal{T}$ means that the associated variety $X_\Sigma$ is a partial compactification of $X_{LT}$, so proving the existence of the triangulation $\mathcal{T}$ is sufficient to prove the Proposition. The second property of $\mathcal{T}$ is not a natural one to consider, but will become necessary to apply results from $\S~\ref{sec:FK}$.

 The proposition can be checked via a simple SAGE program \cite{sagemath} using the TOPCOM package \cite{TOPCOM}, however, we include an explicit proof on how such a triangulation can be constructed.

To prove the proposition, we break the statement up into 3 steps.

\begin{itemize}
    \item[Step 1:] We start by defining an explicit regular polyhedral subdivision $\mathcal{S}$ of $\mathfrak{P}$ containing $\mathcal{T}_0$.
    \item[Step 2:] We prove that the polyhedral subdivision $\mathcal{S}$ can be refined to a regular triangulation $\mathcal{T}$ of $\mathfrak{P}$ containing $\mathcal{T}_0$.
    \item[Step 3:] We show that any regular triangulation obtained this way fulfills the conditions outlined in the Proposition.
\end{itemize}

\subsubsection{Step 1:} We note that $\mathcal{T}_0$ is a regular triangulation of the set of points $P_0,\dots,P_{5},S_1,S_2$. It is in fact a star subdivision with respect to $S_1,S_2$ of the convex hull $\conv(P_0,\dots,P_{5},S_1,S_2)$. Indeed, an example of an explicit weight function $w$ giving the triangulation $\mathcal{T}_0$ is $w(S_1)=w(S_2)=1, w(P_i)=2$ for $0\le i\le 5$. To complete Step 1, we extend this weight function to all $14$ points of $\mathfrak{C}$.

Consider the weight function $w(P_i)=2$ for $0\le i\le 5$, $w(S_1)=w(S_2)=1$ ($j=1,2$) and $w(P_j)=5$ for $6\le j\le 11$. The convex hull of the points 
\[
Z_i=(P_i,w(P_i)),\quad   R_j=(S_j,w(S_j)), \quad (0\le i\le 11, \  j=1,2)
\]then forms a polyhedron $\mathcal{Q}$ in $\R^{8}$. To obtain the regular subdivision of $\mathfrak{P}$ corresponding to the weight function $w$, we need to project the lower facets of the polyhedron $\mathcal{Q}$ down to $\R^{7}$ along the last coordinate. A lower facet is defined to be a facet of $\mathcal{Q}$ where the inward pointing normal has a positive last coordinate.

We claim there are exactly  12 lower facets of $\mathcal{Q}$. We write each lower facet $F_i$ in the form $u_i\cdot x+a_i= 0$ where $u_i$ is the inward pointing normal of the $i^{th}$ facet. Take $H_i$ to be the halfspace corresponding to the lower facet $F_i$, i.e. the halfspace given by $u_i\cdot x+a\geq 0$.
The normals and additive constants are:

\begin{itemize}
    \item $H_0: (5,-1,-1,0,0,0,0,3)x-3\ge0$
    \item $H_1: (-1,5,-1,0,0,0,0,3)x-3\ge0$
    \item $H_2: (-1,-1,5,0,0,0,0,3)x-3\ge0$
    \item $H_{3}: (1,1,1,6,0,0,0,3)x-3\ge0$
    \item $H_{4}: (1,1,1,0,6,0,0,3)x-3\ge0$
    \item $H_{5}: (-5,-5,-5,-6,-6,0,0,3)x-3\ge0$
    \item $H_{6}: (3,-1,-1,0,0,0,2,1)x-1\ge0$
    \item $H_{7}: (-1,3,-1,0,0,0,2,1)x-1\ge0$
    \item $H_{8}: (-1,-1,3,0,0,0,2,1)x-1\ge0$
    \item $H_{9}: (1,1,1,4,0,0,-2,1)x+1\ge0$
    \item $H_{10}: (1,1,1,0,4,0,-2,1)x+1\ge0$
    \item $H_{11}: (-3,-3,-3,-4,-4,0,-2,1)x+1\ge0$.
\end{itemize}

An easy computation shows that all $14$ points lie in the intersection of the relevant half-spaces. This is a direct consequence of the fact that $\mathcal{Q}\subseteq H_i$ for $i=0,\dots,11$.  Table~\ref{Tab:facets} shows which points lie on each lower facet.

\begin{table}
\begin{center}
\begin{tabular}{ |c|c| } 
\hline
  \text{Facet} & \text{contains} \\
 \hline
 $F_0$ & $Z_1,\dots,Z_{5},R_1,R_2$\\
 $\vdots$ & $\vdots$\\
 $F_{5}$ & $Z_0,\dots, Z_{4}, R_1,R_2$\\
 \hline
 $F_{6}$ & $Z_1,\dots,Z_{5},R_1,Z_{7},Z_{8}$ \\
 $F_7$ & $Z_0,Z_2,\dots,Z_5,R_1,Z_6,Z_8$\\
 $F_{8}$ & $Z_0, Z_{1},Z_{3},Z_4,Z_{5},R_1,Z_{6},Z_{7}$ \\
 \hline
 $F_{9}$ & $Z_0,Z_1,Z_2,Z_{4},Z_{5},R_2,Z_{10},Z_{11}$\\
 $F_{10}$ & $Z_0,Z_1,Z_2,Z_3,Z_5,R_2,Z_{9},Z_{11}$\\
 $F_{11}$ & $Z_0,\dots,Z_{4},R_2,Z_{9},Z_{10}$.\\
 \hline
\end{tabular}
\end{center}
\caption{Dictionary of points contained in each lower facet of $\mathcal{Q}$.}
\label{Tab:facets}
\end{table}

To obtain the polyhedral subdivision $\mathcal{S}$ of $\mathfrak{P}$ corresponding to the weight function $w$, we now project these facets down to $\R^{7}$ along the last coordinate. Denoting by $\widehat{F}_i$ the polyhedron obtained by projecting the facet $F_i$, we obtain the set of $12$ polyhedra given in Table~\ref{Tab:PolysinS}.
We note here that when projecting, all points that lied on the facet $F_i$ lie in the polyhedron $\widehat{F}_i$, by convexity of the polyhedron $\mathcal{Q}$ in $\R^{8}$.

\begin{table}[ht!]
\begin{center}
\begin{tabular}{ |ll| } 
\hline
 $\widehat{F}_0=$&$\conv(P_1,\dots,P_{5},S_1,S_2)$\\
 $\vdots$ & $\vdots$\\
 $\widehat{F}_{5}=$&$\conv(P_0,\dots, P_{4}, S_1,S_2)$\\
 \hline
 $\widehat{F}_{6}=$&$\conv(P_1,\dots,P_{5},S_1,P_{7},P_{8})$\\
 $\widehat{F}_7=$&$\conv(P_0,P_2,\dots,P_5,S_1,P_6,P_8)$\\
 $\widehat{F}_{8}=$&$\conv(P_0,P_{1},P_{3},P_4,P_{5},S_1,P_{6},P_{7})$\\
 \hline
 $\widehat{F}_{9}=$&$\conv(P_0,P_1,P_2,P_{4},P_{5},S_2,P_{10},P_{11})$\\
 $\widehat{F}_{10}=$ &$\conv(P_0,\dots,P_3,P_5,S_2,P_{9},P_{11})$  \\
 $\widehat{F}_{11}=$&$\conv(P_0\dots,P_{4},S_2,P_{9},P_{10})$.\\
 \hline
\end{tabular}
\end{center}
\caption{Polyhedra in the regular subdivision.}
\label{Tab:PolysinS}
\end{table}

It remains to show that the above collection $F_i$ contains all the lower facets of $\mathcal{Q}$. 
Showing that there is no other lower facet of $\mathcal{Q}$ apart from $F_0,\dots,F_{11}$ is equivalent to showing that $\bigcup F_i +\langle (0,\dots,0,1)\rangle_{\R_{\geq 0}}$ contains the entire polyhedron $\mathcal{Q}$. Since all vertices of $\mathcal{Q}$ lie inside each half-space $H_i$, it suffices to show that the union of the projections $\widehat{F}_i$ contains the convex hull of $P_0,\dots,P_{11},S_1,S_2$, i.e. contains $\mathfrak{P}$. This is equivalent to saying that they give a polyhedral subdivision (regularity is given by construction).

 So we aim to prove the following claim.

\begin{Lemma}
\label{Lem:Polysubdivision}
For $\widehat{F}_i$ and $\mathfrak{P}$ as above, we have $\bigcup_{i=0}^{11}\widehat{F}_i=\mathfrak{P}$.
\end{Lemma}
\noindent To prove Lemma $\ref{Lem:Polysubdivision}$, we will need the following result.
\begin{Lemma}
\label{Lem:inequalitybounds}
Suppose we are given a set of $m$ inequalities $L_j\le R_j$ with $\sum_{j=1}^m L_j\le C\le \sum_{j=1}^m R_j$, then there exists an $m$-tuple of real numbers $a_j$ such that $L_j\le a_j\le R_j$ and $\sum_{j=1}^m a_j=C$.
\end{Lemma}
\begin{proof}
To show that the claim holds, we define $a_j(x)=L_j+x(R_j-L_j)$. This is a linear function such that, for all $x\in [0,1]$, $L_j\le a_j(x)\le R_j$. Define $f(x)=\sum a_j(x)$. $f$ is itself linear and thus continuous in $x$, with $f(0)=\sum_{j=1}^m L_j\le C\le \sum_{j=1}^m R_j=f(1)$. By the intermediate value theorem, there is an $x_C\in [0,1]$ such that $f(x)=\sum_{j=1}^m a_j(x_C)=C$. Setting $a_j=a_j(x_C)$ gives the $m$-tuple, proving the claim.
\end{proof}

\begin{proof}[Proof of Lemma $\ref{Lem:Polysubdivision}$]
The first thing to note is that \[\mathfrak{P}=\conv(P_0,\dots,P_{11},S_1,S_2)=\conv(P_0,\dots,P_{11}).\] So we will show that $\bigcup_{i=0}^{11}\widehat{F}_i=\conv(P_0,\dots,P_{11})$. 

We start by showing that $\bigcup_{i=0}^{5}\widehat{F}_i=\conv(P_0,\dots,P_{5},S_1,S_2)$, which is equivalent to saying that $\widehat{F}_0,\dots,\widehat{F}_{5}$ form a polyhedral subdivision of $\conv(P_0,\dots,P_{5},S_1,S_2)$.

The inclusion $\subseteq$ is immediate from Table \ref{Tab:PolysinS}, so it remains to check the opposite inclusion. Any point $X\in\conv(P_0,\dots,P_{5},S_1,S_2)$ can be written as $X=\sum_{i=0}^{5} \lambda_i P_i+\mu_1S_1+\mu_2S_2$ for some $\lambda_i,\mu_j\in\R_{\geq0}$ with $\sum\lambda_i+\mu_1+\mu_2=1$. Note also that $\sum_{i=0}^{5}P_i=3(S_1+S_2)$. Now define $j$ such that $\lambda_j=\min_{0\le i\le 5}\{\lambda_i\}$. Then \[X=\sum_{i=0}^{5}(\lambda_i-\lambda_j)P_i+(3\lambda_j+\mu_1)S_1+(3\lambda_j+\mu_2)S_2=\sum_{\substack{0\le i\le 5\\ i\neq j}} (\lambda_i-\lambda_j)P_i+(3\lambda_j+\mu_1)S_1+(3\lambda_j+\mu_2)S_2.\] Since $\lambda_j=\min_{0\le i\le 5}\{\lambda_i\}\le \lambda_i$ for $0\le i\le 5$, we have that $(\lambda_i-\lambda_j)\geq 0$ for $0\le i\le 5$. As $\lambda_i,\mu_1,\mu_2\geq 0$, we also have $3\lambda_j+\mu_1, 3\lambda_j+\mu_2\geq0$. Also, \[\sum_{\substack{0\le i\le 5\\ i\neq j}}(\lambda_i-\lambda_j)+(3\lambda_j+\mu_1)+(3\lambda_j+\mu_2)=\sum_{i=0}^{5}\lambda_i+\mu_1+\mu_2=1,
\]
and thus $X\in \widehat{F}_j$. This shows $\bigcup_{i=0}^{5}\widehat{F}_i=\conv(P_0,\dots,P_{5},S_1,S_2)$.

To show $\bigcup_{i=0}^{11}\widehat{F}_i=\mathfrak{P}$, we note again that the inclusion $\subseteq$ is immediate. For the opposite inclusion $\supseteq$, take a general point $X$ in $\mathfrak{P}$. Then $X$ can be written as $X=\sum_{i=0}^{11}\lambda_i P_i$ with $\lambda_i\geq 0$ for $0\le i\le 11$ and $\sum_{i=0}^{11}\lambda_i=1$. 

Without loss of generality, assume that $(\lambda_6+\lambda_7+\lambda_8)\geq(\lambda_9+\lambda_{10}+\lambda_{11})$ (the case where the inequality is reversed is analogous). We will now show that if $X\not\in\bigcup_{i=0}^{5}\widehat{F}_i=\conv(P_0,\dots,P_{5},S_1,S_2)$, then $X\in\bigcup_{i=6}^{8} \widehat{F}_i$ (if the inequality had been reversed, then $X$ would be in $\bigcup_{i=9}^{11}\widehat{F}_i$).

Let
\begin{equation*}\begin{array}{llr}
     \nu_i= &\lambda_i+\lambda_{6+i}-\frac{1}{3}\left((\lambda_6+\lambda_7+\lambda_8)-(\lambda_9+\lambda_{10}+\lambda_{11})\right)&\text{for }0 \le i\le 2, \\
     \nu_i= &\lambda_i+\lambda_{6+i}&\text{for }3\le i\le 5,\\
     \mu_1=&\left((\lambda_6+\lambda_7+\lambda_8)-(\lambda_9+\lambda_{10}+\lambda_{11})\right),&\\
     \mu_2=&0.& 
\end{array}\end{equation*}

\noindent Then \[
\sum_{i=0}^{5}\nu_i P_i+\mu_1S_1+\mu_2S_2=\sum_{i=0}^{11}\lambda_i P_i
\]
and \[
\sum_{i=0}^{5}\nu_i+\mu_1+\mu_2=\sum_{i=0}^{11}\lambda_i=1.
\]

Note $\mu_1\geq0$ by assumption and $\mu_2=0$. Thus, if $\nu_i\geq 0$ for $0\le i\le 5$,  $X$ is expressed as an element of $\conv(P_0,\dots,P_{5},S_1,S_2)=\bigcup_{i=0}^{5}\widehat{F}_i$ using the above equations. Otherwise, we will claim that $X\in\bigcup_{i=6}^8 \widehat{F_i}$. For $3\le i\le 5$, we have $\nu_i\geq 0$ as both $\lambda_i$ and $\lambda_{6+i}$ are $\geq0$. We turn our attention to the $\nu_i$ for $i=0,1,2$.

For $0\le i\le 2$, $\nu_i\geq0$ is equivalent to \[\frac{1}{3}\left((\lambda_6+\lambda_7+\lambda_8)-(\lambda_9+\lambda_{10}+\lambda_{11})\right)\le \lambda_i+\lambda_{6+i},\] so the condition that all $\nu_i$ are non-negative is equivalent to \[\frac{1}{3}\left((\lambda_6+\lambda_7+\lambda_8)-(\lambda_9+\lambda_{10}+\lambda_{11})\right)\le \min_{0\le i\le 2}\{\lambda_i+\lambda_{6+i}\}.\] 

Therefore, $X\in\bigcup_{i=0}^{5}\widehat{F}_i=\conv(P_0,\dots,P_{5},S_1,S_2)$ if $\frac{1}{3}\left((\lambda_6+\lambda_7+\lambda_8)-(\lambda_9+\lambda_{10}+\lambda_{11})\right)\le \min_{0\le i\le 2}\{\lambda_i+\lambda_{6+i}\}$. Suppose this condition does not hold, i.e. \begin{equation}\label{eq:pftriangulation}\min_{1\le i\le 3}\{\lambda_i+\lambda_{6+i}\}<\frac{1}{3}\left((\lambda_6+\lambda_7+\lambda_8)-(\lambda_9+\lambda_{10}+\lambda_{11})\right).\end{equation}

Without loss of generality, we may assume that $\lambda_0+\lambda_6=\min_{0\le i\le 2}\{\lambda_i+\lambda_{6+i}\}$ (by symmetry, the other cases are analogous). We will show that $X\in\widehat{F}_{6}$. Any point $Y$ in $\widehat{F}_{6}=\conv{(P_1,\dots,P_{5},S_1,P_{7},P_{8})}$ can be written as \[
Y=\sum_{i=1}^{5}\nu_i P_i+\sum_{i=7}^{8}\nu_i P_i+\mu_1S_1.
\]
If we find $\nu_i,\mu_1$ such that this sum is equal to $\sum_{i=0}^{11}\lambda_i P_i=X$, we are done as we will have expressed $X$ as an element of $\widehat{F}_{6}$. 

Given a choice of real numbers $\alpha_1,\alpha_2$ with $\alpha_1+\alpha_2=1$, define

\begin{equation*}\begin{array}{llr}
    \nu_i=& \lambda_i+\alpha_i(3\lambda_0+2\lambda_{6}+(\lambda_9+\lambda_{10}+\lambda_{11}))-(\lambda_0+\lambda_6)&\text{for } 1\le i\le 2,\\
    \nu_i=& \lambda_i+\lambda_{6+i} & \text{for } 3\le i\le 5,\\
    \mu_1=& 3\lambda_0+3\lambda_{6},& \\
     \nu_{6+i}=& \lambda_{6+i}+\alpha_i(-3\lambda_0-2\lambda_{6}-(\lambda_9+\lambda_{10}+\lambda_{11})) & \text{for }1\le i\le 2.
\end{array}\end{equation*}
Substituting these values into the expression for $Y$ gives
\[
Y=\sum_{i=1}^{5}\nu_i P_i+\sum_{i=7}^{8}\nu_i P_i+\mu_1S_1=\sum_{i=0}^{11}\lambda_i P_i=X,
\]
as well as \[
\sum \nu_i+\mu_1=\sum_{i=0}^{11}\lambda_i=1.
\]
For this choice of $\nu_i$'s and $\mu_1$ to define an element $Y\in \widehat{F}_{6}$, we require $\nu_i\geq 0$ for all $i$ and $\mu_1\geq 0$. We note that, as $\lambda_0,\lambda_{6}\geq 0$, we have $\mu_1\geq 0$.

Therefore, what remains to prove is that there exist $\alpha_1,\alpha_2\in \R$ with $\alpha_1+\alpha_2=1$ such that $\nu_i\geq 0$ for $i\in\{1,\dots,5,7,8\}$.
For $i=1,2$, we can arrange the inequalities $\nu_i\geq 0$ and $\nu_{6+i}\geq 0$ to give
\begin{equation}
\label{eq:ineqstriang}
\frac{\lambda_0+\lambda_6-\lambda_i}{3\lambda_0+2\lambda_{6}+(\lambda_9+\lambda_{10}+\lambda_{11})}\le \alpha_i\le \frac{\lambda_{6+i}}{3\lambda_0+2\lambda_{6}+(\lambda_9+\lambda_{10}+\lambda_{11})}.
\end{equation}
This works provided $3\lambda_0+2\lambda_{6}+(\lambda_9+\lambda_{10}+\lambda_{11})\neq 0$ but if that term was zero, then by non-negativity of the $\lambda_i$ we would have $\lambda_0=\lambda_{6}=\lambda_{9}=\dots=\lambda_{11}=0$ and thus $X\in\widehat{F}_{6}$.
So if there exists a pair $(\alpha_1,\alpha_2)$ with $\eqref{eq:ineqstriang}$ holding for $i=1,2$ and $\alpha_1+\alpha_2=1$, then $X\in \widehat{F}_{6}$.

\noindent Note that for all $i$, \[
\frac{\lambda_0+\lambda_6-\lambda_i}{3\lambda_0+2\lambda_{6}+(\lambda_9+\lambda_{10}+\lambda_{11})}\le\frac{\lambda_{6+i}}{3\lambda_0+2\lambda_{6}+(\lambda_9+\lambda_{10}+\lambda_{11})},\]
as $\lambda_0+\lambda_6=\min_{0\le i\le 2}\{\lambda_i+\lambda_{6+i}\}$.

\noindent Furthermore, 
\begin{equation*}
\begin{array}{lrl}
  & 0 & \le \lambda_0+\lambda_1+\lambda_2+\lambda_9+\lambda_{10}+\lambda_{11}\\
 \Leftrightarrow & 2\lambda_0+2\lambda_{6} & \le 3\lambda_0+\lambda_1+\lambda_2+2\lambda_{6}+\lambda_9+\lambda_{10}+\lambda_{11}\\
 \Leftrightarrow & \sum_{i=1}^2 \lambda_0+\lambda_6-\lambda_i & \le 3\lambda_0+2\lambda_{6}+(\lambda_9+\lambda_{10}+\lambda_{11})\\
 \Leftrightarrow & \sum_{i=1}^{2}\frac{\lambda_0+\lambda_6-\lambda_i}{3\lambda_0+2\lambda_{6}+(\lambda_9+\lambda_{10}+\lambda_{11})} & \le 1.
\end{array}
\end{equation*}
Lastly, we are given that $\lambda_0+\lambda_6=\min_{0\le i\le 2}\{\lambda_i+\lambda_{6+i}\}\le \frac{1}{3}\left((\lambda_6+\lambda_7+\lambda_8)-(\lambda_9+\lambda_{10}+\lambda_{11})\right)$. This leads to the following sequence of implications:

\begin{equation*}
\begin{array}{lrl}
& \lambda_0+\lambda_6 &\le \frac{1}{3}\left((\lambda_6+\lambda_7+\lambda_8)-(\lambda_{9}+\lambda_{10}+\lambda_{11})\right)\\
\Leftrightarrow & 3\lambda_0+2\lambda_{6}+\lambda_9+\lambda_{10}+\lambda_{11} & \le \lambda_{7}+\lambda_{8}\\
\Leftrightarrow & 1 & \le \sum_{i=1}^{2}\frac{\lambda_{6+i}}{3\lambda_0+2\lambda_{6}+(\lambda_9+\lambda_{10}+\lambda_{11})}.
\end{array}
\end{equation*}

In summary, we have shown that for $i=1,2$, we have \[\frac{\lambda_0+\lambda_6-\lambda_i}{3\lambda_0+2\lambda_{6}+(\lambda_9+\lambda_{10}+\lambda_{11})}\le\frac{\lambda_{6+i}}{3\lambda_0+2\lambda_{6}+(\lambda_9+\lambda_{10}+\lambda_{11})},\]
 and that 
 
 \[\sum_{i=1}^{2}\frac{\lambda_0+\lambda_6-\lambda_i}{3\lambda_0+2\lambda_{6}+(\lambda_9+\lambda_{10}+\lambda_{11})}\le 1\le \sum_{i=1}^{2}\frac{\lambda_{6+i}}{3\lambda_0+2\lambda_{6}+(\lambda_9+\lambda_{10}+\lambda_{11})}.\]

Applying Lemma \ref{Lem:inequalitybounds} gives us the existence of a pair $\alpha_1,\alpha_2$ as required, concluding the proof to Lemma \ref{Lem:Polysubdivision}.
\end{proof}

The Lemma \ref{Lem:Polysubdivision} shows that we have indeed found all lower facets of the polyhedron $\mathcal{Q}$, meaning that the collection $\widehat{F}_1,\dots,\widehat{F}_{11}$ gives a regular polyhedral subdivision $\mathcal{S}$ of $\mathfrak{P}$, thus concluding Step 1.

\subsubsection{Step 2:} 
This is true by general convex geometry (using the poset of refinements and the secondary polytope). By Theorem 2.4 in Chapter 7 of \cite{GKZ}, the poset of (non-empty) faces of the secondary polytope $\Sigma(\mathfrak{P})$ is isomorphic to the poset of all regular subdivisions of $\mathfrak{P}$, partially ordered by refinement (see also Theorem 16.4.1 in \cite{Handbook}). The vertices of $\Sigma(\mathfrak{P})$ correspond to regular triangulations. Thus, our regular subdivision obtained by projection must correspond to some face of $\Sigma(\mathfrak{P})$ and any vertex of that face will correspond to a regular triangulation refining it.

\subsubsection{Step 3:}
Consider a regular triangulation $\mathcal{T}$ obtained by refining $\mathcal{S}$. By definition, it is a regular triangulation of $\mathfrak{P}$. Recall Table \ref{Tab:PolysinS}. Denote by $C_i$ the collection of points used to define the polyhedron $\widehat{F}_i$ in the table. Note that $\widehat{F}_0,\dots,\widehat{F}_{5}$ are the simplices in $\mathcal{T}_0$, and therefore any simplices in $\mathcal{T}\setminus\mathcal{T}_0$ do not originate from refining any of $\widehat{F}_0,\dots,\widehat{F}_5$. 

Thus the last step of the proof reduces to showing that none of the polyhedra $\widehat{F}_i,\ 0\le i \le 11,$ contain any of the points we did not define it by, i.e. $\widehat{F}_i\cap\mathfrak{C}=C_i$. Indeed, in that case we note that, by consulting Table \ref{Tab:PolysinS}, the polyhedra $\widehat{F}_i$ each fulfill at least one of the conditions $A$ or $B$ in the proposition. If $\widehat{F}_i\cap\mathfrak{C}=C_i$, then all simplices in a refinement of $\widehat{F}_i$ are defined as the convex hull of a subset of $C_i$ (as there is no interior point to refine upon), thus inheriting the properties $A$ or $B$ from $\widehat{F}_i$. 

Showing that $\widehat{F}_i\cap\mathfrak{C}=C_i$ for $6\le i\le 11$ reduces to a simple computation. We shall do the computation for $\widehat{F}_6$, as the remaining cases are analogous by symmetry. 

We need to show that $P_0,P_6,P_9,P_{10},P_{11},S_1\not \in \widehat{F}_6$. Any point $X$ in $\widehat{F}_6$ can be written as \begin{eqnarray}\label{eq:Pinhats}
&&\lambda_1 P_1+\dots+\lambda_5 P_5+\mu_1 S_1+\lambda_7 P_7+\lambda_8 P_8=(
-\lambda_3-\lambda_4-\lambda_5-\lambda_7-\lambda_8,\nonumber\\
 &&3\lambda_1-\lambda_3-\lambda_4-\lambda_5+2\lambda_7-\lambda_8,
3\lambda_2-\lambda_3-\lambda_4-\lambda_5-\lambda_7+2\lambda_8,
-\lambda_1-\lambda_2+3\lambda_3,\nonumber\\
&&\lambda_1+\lambda_2+3\lambda_4,
\lambda_3+\lambda_4+\lambda_5+\mu_1+\lambda_7+\lambda_8,
\lambda_1+\lambda_2),
\end{eqnarray} with $\lambda_i,\mu_1\ge 0$ and $\sum \lambda_i+\mu_1=1$. We note that the last two coordinates of $X$ are $\lambda_3+\lambda_4+\lambda_5+\mu_1+\lambda_7+\lambda_8$ and $\lambda_1+\lambda_2$ respectively.
Assume $P_0\in\widehat{F}_6$ and had an expression as in Equation \eqref{eq:Pinhats}. Then, as $\lambda_i,\mu_1\geq0$, we can see by looking at the last two coordinates that $\lambda_3=\lambda_4=\lambda_5=\mu_1=\lambda_7=\lambda_8=0$  and $\lambda_1+\lambda_2=1$. But then the first coordinate is $\lambda_1\cdot 0+\lambda_2\cdot 0=0\neq 2$, hence we get a contradiction and $P_0\not\in \widehat{F}_6$. By an analogous reasoning, for $S_2,P_9,P_{10},P_{11}$ we obtain that all but $\lambda_1,\lambda_2$ would need to be 0 again and the sum of these two would need to be 1, which means that not both the second and third coordinate (being $3\lambda_1, 3\lambda_2$) can be 0. Hence $S_2,P_9,P_{10},P_{11}\not\in\widehat{F}_6$. 

Finally, we need to show $P_6\not\in\widehat{F}_6$. Assume we had an expression for $P_6$ as in Equation \eqref{eq:Pinhats}. Since $\lambda_i\geq 0$, considering the last two coordinates gives $\lambda_1=\lambda_2=0$ (since $\lambda_i\geq 0$) and $\lambda_3+\lambda_4+\lambda_5+\mu_1+\lambda_7+\lambda_8=1$. But then the first coordinate is $-(\lambda_3+\lambda_4+\lambda_5+\lambda_7+\lambda_8)\le 0<2$, a contradiction. Thus $P_6\not\in\widehat{F}_6$, and thus $\widehat{F}_6\cap\mathfrak{C}=C_6$ as claimed.

The other cases are analogous by symmetry. Thus we finished Step 3, hence proving Proposition \ref{Prop:LTpartiallycpt}.

\subsection{The ideals associated to the partial compactification}

Recalling the notation from \textsection \ref{sec:FK}, we denote by $x_i$ the variable in $\C[x_0,\dots,x_{11},u_1,u_2]$ corresponding to the point $P_i$ and by $u_j$ the variable corresponding to $S_j$. These 14 variables correspond to the rays of the fan $\Sigma_{\nabla,D_a',D_a'}$ from Corollary \ref{Cor:TotXnabla}. 

\begin{Lemma}
\label{Lem:sectioninBB}
There exists a global function on $X_{\nabla,D_a',D_b'}$ that has the form
\begin{equation}\begin{aligned}\label{eq:SuperpotentialLT}w&=u_1(c_0x_0^3x_6^3+c_1x_1^3x_7^3+c_2x_2^3x_8^3-3\lambda_1x_3x_4x_5x_6x_7x_8) \\& \qquad {} +u_2(c_3x_3^3x_9^3+c_4x_4^3x_{10}^3+c_5x_5^3x_{11}^3-3\lambda_2x_0x_1x_2x_9x_{10}x_{11}),
\end{aligned}\end{equation}
for some $c_i,\lambda_j\in\C$.
\end{Lemma}
\begin{proof}

Consider the hyperplane 
\begin{equation}
    H := \{ (m, t_1, t_2) \in M_{\mathbb{R}} \oplus \mathbb{R}^2 \ | \ t_1 + t_2 = 1\}
\end{equation}
in $M_{\mathbb{R}} \oplus \mathbb{R}^2$. The cone $|\Sigma_{\nabla,D_a',D_b'}|^\vee$ is given by the cone the over the convex hull of the following 8 points on $H$: 

\begin{equation*}\begin{array}{llll}
( 0,  0,  0,  0,  1, 0, 1), & (-1, -1, -1, -1, -1, 0, 1), & ( 0,  0,  0,  1,  0, 0, 1), & ( 0,  0,  1,  0,  0, 1, 0), \\ ( 0,  1,  0,  0,  0, 1, 0), & ( 0,  0,  0,  0,  0, 0, 1), & ( 1,  0,  0,  0,  0, 1, 0), & ( 0,  0,  0,  0,  0, 1, 0).\;
\end{array}\end{equation*}
Recall that there is a correspondence between points in the dual cone $|\Sigma_{LT,D_a,D_b}|^\vee$ and global functions on $X_{LT,D_a,D_b}$. Note that, when one constructs a superpotential for $X_{LT,D_a,D_b}$ using this correspondence with the 8 points above, one obtains the superpotential $w=u_1Q_{1,\lambda}+u_2Q_{2,\lambda}$. To see the global function on $X_{\nabla,D_a',D_b'}$, it suffices to compute what monomials these 8 points correspond to on $X_{\nabla,D_a',D_b'}$. 
This gives the 8 monomials in $\eqref{eq:SuperpotentialLT}$.
\end{proof}

For our purposes of comparing the Batyrev-Borisov construction with the one by Libgober-Teitelbaum, we choose $c_i=1$ and $\lambda_1=\lambda_2=:\lambda$.

We fix a triangulation $\mathcal{T}$ fulfilling the properties of Proposition \ref{Prop:LTpartiallycpt}.
 Let $X=\C^{14}$ and consider the group $G_\Sigma$ corresponding to the fan $\Sigma_{\nabla,D_a',D_b'}$ with its action on $X$.  
From the triangulation $\mathcal{T}$ we obtain the ideals:

\begin{eqnarray*}
&& \I:=\left\langle \prod_{i\not\in I}x_i\prod_{j\not\in J}u_j\ \Big\vert\  \bigcup_{i\in I}u_{\overline{\rho}_i}\cup\bigcup_{j\in J}u_{\tau_j}\text{ give the set of vertices of a simplex in }\mathcal{T}\right\rangle,\\
&& \J:=\left\langle \prod_{i\not\in I}x_i\ \Big\vert\  \bigcup_{i\in I}u_{\overline{\rho}_i}\cup\bigcup_{j=1}^2u_{\tau_j}\text{ give the set of vertices of a simplex in }\mathcal{T}\right\rangle.\;
\end{eqnarray*}

Before we can apply Proposition $\ref{Prop:FK19Prop4.7}$ and Corollary \ref{Cor:FK19Cor4.8}, we need to ensure the condition $\I\subseteq\sqrt{\partial w,\J}$ holds.
\begin{Lemma}\label{Lemma:LTIdealcontainment}
For any triangulation $\mathcal{T}$ as in Proposition \ref{Prop:LTpartiallycpt}, defining $\I,\J$ and $w$ as above with $\lambda^6\ne 0,1$, we have $\I\subseteq\sqrt{\J,\partial w}$. Therefore, this choice of superpotential fulfills the condition of Proposition $\ref{Prop:FK19Prop4.7}$.
\end{Lemma}
\begin{proof}

To show the containment $\I\subseteq\sqrt{\partial w,\J}$, we prove that all the generators of $\I$ are in $\sqrt{\partial w,\J}$.
The ideal $\I$ is, by definition, generated by the monomials which correspond to the simplices in the triangulation $\mathcal{T}$. For a simplex $T\in\mathcal{T}_0$, both $S_1$ and $S_2$ are vertices. Thus, by definition, the monomial associated to $T$ is in $\J$ and hence in $\sqrt{\partial w,\J}$.

For any simplex $T\in \mathcal{T}\setminus\mathcal{T}_0$, either condition $(A)$ or $(B)$ of Proposition \ref{Prop:LTpartiallycpt} holds. We claim that the monomial associated to a simplex $T$ fulfilling either of those two conditions is an element of $\sqrt{\partial w}$ and therefore an element of $\sqrt{\partial w, \J}$.

We note that if $T\in\mathcal{T}\setminus\mathcal{T}_0$ fulfills condition $(A)$, i.e. does not contain any of the points $S_1.P_6,P_7,P_8$ and there is a pair of points of the form $P_j,P_{6+j}$ with $3\le j\le 5$ also not contained, then by definition $u_1x_jx_{6+j}x_6x_7x_8$ divides the monomial associated to $T$. Similarly, if $T$ fulfilled condition $(B)$ instead, $u_2x_jx_{6+j}x_9x_{10}x_{11}$ (for some $0\le j\le 2$) would divide the monomial generator of $\I$ associated to $T$.

To show that any monomial associated to a simplex in $\mathcal{T}\setminus\mathcal{T}_0$ is in $\sqrt{\partial w}\subseteq\sqrt{\partial w,\J}$, it is thus sufficient to prove that the six monomials $u_2x_0x_{6}x_9x_{10}x_{11},$  $u_2x_1x_{7}x_9x_{10}x_{11},$ $u_2x_2x_{8}x_9x_{10}x_{11},$ $u_1x_3x_{9}x_6x_7x_8,$ $u_1x_4x_{10}x_6x_7x_8$ and $u_1x_5x_{11}x_6x_7x_8$ are elements of $\sqrt{\partial w}$.

By symmetry of the $x_i$ in $w$, we note that it is sufficient to show that $u_2x_0x_6x_9x_{10}x_{11}\in\sqrt{\partial w}$.
Start by explicitly writing down the ideal $\langle\partial w\rangle$, i.e. the ideal generated by the partial derivatives of $w$.

\begin{equation*}\begin{array}{ll}
\langle \partial w\rangle=& \langle 3u_1x_0^2x_6^3-3\lambda u_2x_1x_2x_9x_{10}x_{11}, 3u_1x_1^2x_7^3-3\lambda u_2x_0x_2x_9x_{10}x_{11},\\ & 3u_1x_2^2x_8^3-3\lambda u_2x_0x_1x_9x_{10}x_{11},
 3u_2x_3^2x_9^3-3\lambda u_1x_4x_5x_6x_7x_8,\\ & 3u_2x_4^2x_{10}^3-3\lambda u_1x_3x_5x_6x_7x_8, 3u_2x_5^2x_{11}^3-3\lambda u_1x_3x_4x_6x_7x_8,\\
& 3u_1x_0^3x_6^2-3\lambda u_1x_3x_4x_5x_7x_8, 3u_1x_1^3x_7^2-3\lambda u_1x_3x_4x_5x_6x_8,\\  &  3u_1x_2^3x_8^2-3\lambda u_1x_3x_4x_5x_6x_7,
 3u_2x_3^3x_9^2-3\lambda u_2x_0x_1x_2x_{10}x_{11},\\ & 3u_2x_4^3x_{10}^2-3\lambda u_2x_0x_1x_2x_9x_{11}, 3u_2x_5^3x_{11}^2-3\lambda u_2x_0x_1x_2x_9x_{10},\\
& x_0^3x_6^3+x_1^3x_7^3+x_2^3x_8^3-3\lambda x_3x_4x_5x_6x_7x_8,  x_3^3x_9^3+x_4^3x_{10}^3+x_5^3x_{11}^3-3\lambda x_0x_1x_2x_9x_{10}x_{11}\rangle\;
\end{array}\end{equation*}

 We see that $3u_1x_i^2x_{6+i}^3-3\lambda u_2\frac{x_0x_1x_2x_9x_{10}x_{11}}{x_i} \in \langle \partial w\rangle$ for $0\le i\le 2$. Notice that since $ac-bd=c(a-b)+b(c-d)$, if $a-b,c-d$ are elements in an ideal, then so is $ac-bd$.
 Hence by iterating this we obtain that \[
27 u_1^3x_0^2x_1^2x_2^2x_6^3x_7^3x_8^3-27\lambda^3u_2^3x_0^2x_1^2x_2^2x_9^3x_{10}^3x_{11}^3\in\langle\partial w\rangle.
\]
Similarly,
\[
27u_2^3x_3^2x_4^2x_5^2x_9^3x_{10}^3x_{11}^3-27\lambda^3u_1^3x_3^2x_4^2x_5^2x_6^3x_7^3x_8^3\in\langle \partial w\rangle.
\]
Therefore,
\begin{eqnarray}\label{eqn:IdealProof1}
&& (27)^2u_1^3u_2^3x_0^2\dots x_5^2x_6^3\dots x_{11}^3-(27)^2\lambda^6 u_1^3u_2^3x_0^2\dots x_5^2x_6^3\dots x_{11}^3\in\langle \partial w\rangle\nonumber\\
&& \Rightarrow 27^2(1-\lambda^6)u_1^3u_2^3x_0^2\dots x_5^2x_6^3\dots x_{11}^3\in\langle\partial w\rangle\nonumber\\
&& \Rightarrow u_1^3u_2^3x_0^2\dots x_5^2x_6^3\dots x_{11}^3\in\langle\partial w\rangle\nonumber\\
&& \Rightarrow (u_1u_2x_0\dots x_{11})^3\in\langle \partial w\rangle.\;
\end{eqnarray}

Consider $\frac{\partial w}{\partial u_1}$, giving
\begin{equation}\label{eq:IdealProof2}
    x_0^3x_6^3+x_1^3x_7^3+x_2^3x_8^3-3\lambda x_3x_4x_5x_6x_7x_8\in\langle \partial w\rangle.
\end{equation}
Furthermore, we note that $\sum_{i=0}^2 x_i\frac{\partial w}{\partial x_i}\in\langle \partial w\rangle$, and thus
\begin{equation}\label{eq:IdealProof3}
    \sum_{i=0}^2(3u_1x_i^3x_{6+i}^3-3\lambda u_2x_0x_1x_2x_9x_{10}x_{11})\in\langle\partial w\rangle.
\end{equation}
By \eqref{eq:IdealProof2} we have that $x_3x_4x_5x_6x_7x_8 + \langle \partial w\rangle = \frac{1}{3\lambda} (x_0^3x_6^3 + x_1^3 x_7^3 + x_2^3x_8^3) + \langle \partial w\rangle $. We use this to substitute into \eqref{eqn:IdealProof1} to obtain that
\[
 \frac{1}{27\lambda^3} u_2^3x_0^3x_1^3x_2^3x_9^3x_{10}^3x_{11}^3(u_1(x_0^3x_6^3+x_1^3x_7^3+x_2^3x_8^3))^3\in\langle \partial w\rangle
\]

Performing the same style of substitution with \eqref{eq:IdealProof3}, we obtain

\[
 u_2^3x_0^3x_1^3x_2^3x_9^3x_{10}^3x_{11}^3u_2^3x_0^3x_1^3x_2^3x_9^3x_{10}^3x_{11}^3=(u_2x_0x_1x_2x_9x_{10}x_{11})^6\in\langle\partial w\rangle
 \]
Thus $u_2x_0x_1x_2x_9x_{10}x_{11}\in\sqrt{\partial w}$.

By comparing the elements $x_2\frac{\partial w}{\partial x_2}, u_2x_0x_1x_2x_9x_{10}x_{11}\in\sqrt{\partial w}$, we obtain that $u_1x_2^3x_8^3\in\sqrt{\partial w}$, implying $u_1x_2x_8\in\sqrt{\partial w}$. This, in turn, implies that $u_2x_0x_1x_9x_{10}x_{11}\in\sqrt{\partial w}$, by inspection of $\frac{\partial w}{\partial x_2}\in\langle\partial w\rangle\subseteq \sqrt{\partial w}$. Similarly, $u_2x_0x_2x_9x_{10}x_{11}\in\sqrt{\partial w}$. We also have $u_2x_5x_{11}\in\sqrt{\partial w}$ by an analogous computation. Finally, $\frac{\partial w}{\partial u_1}=x_0^3x_6^3+x_1^3x_7^3+x_2^3x_8^3-3\lambda x_3x_4x_5x_6x_7x_8\in\sqrt{\partial w}$.

Therefore, one can intuit and then compute that

\begin{equation}\begin{aligned}
 (u_2x_0x_6x_9x_{10}x_{11})^4 &= -u_2^3x_1^2x_6x_7^3x_9^3x_{10}^3x_{11}^3\cdot(u_2x_0x_1x_9x_{10}x_{11}) \\
    & \quad - u_2^3x_2^2x_6x_8^3x_9^3x_{10}^3x_{11}^3\cdot (u_2x_0x_2x_9x_{10}x_{11}) \\
    & \quad + 3\lambda u_2^3x_0x_3x_4x_6^2x_7x_8x_9^4x_{10}^4x_{11}^4\cdot (u_2x_5x_{11}) \\
    &\quad + u_2^4x_0x_6x_9^4x_{10}^4x_{11}^4\cdot (x_0^3x_6^3+x_1^3x_7^3+x_2^3x_8^3-3\lambda x_3x_4x_5x_6x_7x_8)\in  \sqrt{\partial w} 
\end{aligned}    
\end{equation}

Thus, we have shown that $u_2x_0x_6x_9x_{10}x_{11}\in\sqrt{\partial w}$.  By symmetry, any simplex fulfilling properties $(A)$ or $(B)$ corresponds to a monomial in $\sqrt{\partial w,\J}$. Hence any monomial associated to a simplex $T\in\mathcal{T}\setminus\mathcal{T}_0$ is an element of $\sqrt{\partial w,\J}$, concluding the proof that $\I\subseteq\sqrt{\partial w,\J}$.
\end{proof}

\begin{Corollary}\label{Cor:ChamberLT}
Consider the GKZ fan of $\tot\left(\O_{X_\nabla}(-D_b')\oplus\O_{X_\nabla}(-D_a')\right)$ and the group $G_\Sigma$ from above. There is a chamber $\sigma_p$ with affine open $U_p$ such that:
\begin{enumerate}
    \item[(i)] $[U_p/G_{\Sigma}]$ is a partial compactification of $\tot\left(\O_{\mathcal{X}_{LT}}(-D_b)\oplus\O_{\mathcal{X}_{LT}}(-D_a)\right)$.
    \item[(ii)] There is a superpotential corresponding to the eight points in $|\Sigma_{\nabla,D_a',D_b'}|^\vee\cap H$ taking the form $w=u_1(x_0^3x_6^3+x_1^3x_7^3+x_2^3x_8^3-3\lambda x_3x_4x_5x_6x_7x_8)+u_2(x_3^3x_9^3+x_4^3x_{10}^3+x_5^3x_{11}^3-3\lambda x_0x_1x_2x_9x_{10}x_{11})$.
    \item[(iii)] With $\I_p,\J_p$ as defined in $\S \ref{sec:FK}$, we have $\I_p\subseteq\sqrt{\partial w,\J_p}$.
\end{enumerate}
\end{Corollary}
\begin{proof}
Proposition $\ref{Prop:LTpartiallycpt}$ proves $(i)$, Lemma $\ref{Lem:sectioninBB}$ proves $(ii)$ and finally Lemma $\ref{Lemma:LTIdealcontainment}$ shows $(iii)$.
\end{proof}

\subsection{Relating $X_{\nabla}$ and $X_{LT}$}

Recall that the partial compactification of the total space $\tot\left(\O_{X_{LT}}(-D_b)\oplus\O_{X_{LT}}(-D_a)\right)$ in Corollary $\ref{Cor:ChamberLT}$ corresponds to a chamber $\sigma_p$ of the GKZ fan of $\tot(\O_{X_\nabla}(-D_b')\oplus\O_{X_\nabla}(-D_a'))$. We then know that it is birationally equivalent to $\tot(\O_{X_\nabla}(-D_b')\oplus\O_{X_\nabla}(-D_a'))$. 

Thus we want to now explicitly find a triangulation of $\mathfrak{P}$ corresponding to the Batyrev-Borisov mirror family. There, the  superpotential will take the form $$w=u_1(x_0^3x_6^3+x_1^3x_7^3+x_2^3x_8^3-3\lambda x_3x_4x_5x_6x_7x_8)+u_2(x_3^3x_9^3+x_4^3x_{10}^3+x_5^3x_{11}^3-3\lambda x_0x_1x_2x_9x_{10}x_{11}).$$ Note that, by Lemma \ref{Lem:sectioninBB}, this is the form the superpotential should take in the Batyrev-Borisov mirror. In other words, we need a chamber $\sigma_q$ in the GKZ fan corresponding to $\tot(\O_{X_\nabla}(-D_b')\oplus\O_{X_\nabla}(-D_a'))$, where a general section of $\O_{X_\nabla}(-D_b')\oplus\O_{X_\nabla}(-D_a')$ will yield a complete intersection in $X_\nabla$, and thus a Batyrev-Borisov mirror.

\begin{Lemma}\label{Lem:TriangExistsForBB}
Consider the GKZ fan of $\tot\left(\O_{X_\nabla}(-D_b')\oplus\O_{X_\nabla}(-D_a')\right)$ and recall the group $G_\Sigma$ from above. There is a chamber $\sigma_q$ with affine open $U_q$ such that: \begin{enumerate}
    \item[(i)] $[U_q/G_\Sigma]=\tot\left(\O_{\mathcal{X}_\nabla}(-D_b')\oplus\O_{\mathcal{X}_\nabla}(-D_a')\right)$.
    \item[(ii)] A superpotential corresponding to the eight lattice points of $|\Sigma_{\nabla,D_a',D_b'}|^\vee\cap H$ is of the form $
    w=u_1(x_0^3x_6^3+x_1^3x_7^3+x_2^3x_8^3-3\lambda x_3x_4x_5x_6x_7x_8)+u_2(x_3^3x_9^3+x_4^3x_{10}^3+x_5^3x_{11}^3-3\lambda x_0x_1x_2x_9x_{10}x_{11})$.
    \item[(iii)] For $\I_q,\J_q$ as defined in $\S\ref{sec:FK}$, $\I_q\subseteq\sqrt{\partial w,\J_q}$.
\end{enumerate}

\end{Lemma}
\begin{proof}
This proof will construct the triangulation $\mathcal{T}_q$ corresponding to the chamber $\sigma_q$. 
We consider the 42 maximal cones from Table $\ref{Tab:irrel}$. For each of those cones $ \sigma_i, 1\le i\le 42$, we associate a simplex given as convex hull of the 5 vertices corresponding to the 5 rays of $\sigma_i$ plus the two vertices corresponding to the bundle coordinates, i.e. $(0,0,0,0,0,1,0)$ and $(0,0,0,0,0,0,1)$. So for example the first cone, with rays $\rho_0,\rho_1,\rho_2,\rho_9,\rho_{10}$, will correspond to the simplex with vertices $(3,0,0,-1,-1,1,0)$, $(0,3,0,-1,-1,1,0)$, $(0,0,3,-1,-1,1,0)$, $(0,0,0,0,0,1,0)$, $(0,0,0,0,0,0,1)$, $(0,0,0,2,-1,1,0)$, $(0,0,0,-1,2,1,0)$. 
Another way to formulate this is that we take the star subdivision of the cones from Table $\ref{Tab:irrel}$ on the two bundle points $S_1,S_2$.

Regularity of this triangulation of the 14 points is an easy consequence of its construction as a star subdivision, hence it corresponds to some chamber $\sigma_q$ in the GKZ-fan. Indeed, the star subdivision can be obtained by giving the points $S_1,S_2$ a weight of 1 and giving all other points the same weight of $w=2$ and then refining the resulting regular polyhedral subdivision into a triangulation. Alternatively, one can check the regularity of this triangulation by using \texttt{SAGE}.

The third item follows from the fact that we do not partially compactify, hence $\J_q=\I_q$ and therefore $\I_q\subseteq\sqrt{\partial w,\J_q}$, as required.
\end{proof}

We now have all the necessary tools to prove the main result of this paper, Theorem \ref{Thm:MyThmInIntro}.

\begin{proof}[Proof of Theorem \ref{Thm:MyThmInIntro}]
Recall the chambers $\sigma_p$ and $\sigma_q$ in the GKZ fan of the toric variety $\tot\left(\O_{X_\nabla}(-D_b')\oplus\O_{X_\nabla}(-D_a')\right)$ given in Corollary $\ref{Cor:ChamberLT}$ and Lemma $\ref{Lem:TriangExistsForBB}$. By applying Corollary $\ref{Cor:FK19Cor4.8}$, we have $\dbcoh{\mathcal{Z}_{\lambda}}\cong\dbcoh{\mathcal{V}_{LT,\lambda}}$, as required.
\end{proof}

We note that analogous computations to the ones displayed in this paper can yield the following result in lower dimension.
\begin{Theorem}\label{Thm:n=2}
Let $Q_{1}=x_1^2+x_2^2-x_3x_4,\ Q_{2}=x_3^2+x_4^2-x_1x_2$ and let $p_1=x_1^2x_5^2+x_2^2x_6^2-x_3x_4x_5x_6,\ p_2=x_3^2x_7^2+x_4^2x_8^2-x_1x_2x_7x_8$. We define the group $G_4\subseteq PGL(3,\C)$ given by the four automorphisms
\[
\operatorname{diag}(1,1,1,1),\  \operatorname{diag}(\zeta_{8},-\zeta_{8},-\zeta_{8}^{-1},\zeta_{8}^{-1}),\  \operatorname{diag}(\zeta_4,\zeta_4,\zeta_4^{-1},\zeta_4^{-1}),\  \operatorname{diag}(\zeta_8^3,-\zeta_8^3,-\zeta_8^{-3},\zeta_8^{-3}),
\]
 where $\zeta_k$ is a primitive $k^{th}$ root of unity. 

The Batyrev-Borisov mirror to $Z(Q_{1},Q_{2})\subseteq \P^3$ can be computed to be a complete intersection $\mathcal{Z}_2$ in a 3-dimensional toric stack $\mathcal{X}_{BB}$ given as the zero locus $Z_2=Z(p_1,p_2)\subseteq \mathcal{X}_{BB}$.
Take the stacky complete intersection $\mathcal{V}_2 := Z(Q_{1},Q_{2})\subseteq [(\C^4\setminus\{0\})/(\C^\ast\times G_4)]$. Then\[
\dbcoh{\mathcal{V}_2}\cong\dbcoh{\mathcal Z_2}.
\]
\end{Theorem}

\begin{remark}\label{Rem:Gen}
One can aim to generalise this to higher dimensions by looking at the zero-set of the two polynomials $$Q_{1,n}=x_1^n+\dots+x_n^n-x_{n+1}\dots x_{2n} \text{ and } Q_{2,n}=x_{n+1}^n+\dots+x_{2n}^n-x_1\dots x_n$$ in $\P^{2n-1}$.

Unfortunately, $Z(Q_{1,n},Q_{2,n})\subseteq \P^{2n-1}$ is itself singular for $n\geq 4$, which poses problems for the required ideal containment condition $\I\subseteq\sqrt{\partial w,\J}$ to hold. However, using these methods of VGIT is still interesting in the context of categorical resolutions. Indeed, the direct generalisation of the Libgober-Teitelbaum construction above can be categorically resolved. This technique and its generalisations are a subject of future work.
\end{remark}

\begin{remark}
The notion of $f$-duality introduced by Rossi in \cite{Rossi20} and \cite{Rossi21} gives an efficient method of computing and extending the Batyrev-Borisov mirror construction. In particular, applying $f$-duality to the variety $V_{LT,\lambda}\subseteq \P^5/G_{81}$ yields $V_{\lambda}\subseteq\P^5$.

The generalisations looked at in the Remark \ref{Rem:Gen} were inspired by $f$-duality and it seems to be an interesting question when, in general, one can use the methods of variations of GIT employed in this paper to strengthen the notion of $f$-duality.
\end{remark}
\bibliography{LT}

\end{document}